%% file: main.tex
\documentclass[a4paper, 12pt]{article}
\pdfoutput=1

\input{_preamble}

\input{_commands}

\title{Mackey formulas for smooth stacks}
\author{Chenjing Bu \quad and \quad Lucien Hennecart}
\date{}

\begin{document}

\initlengths

\maketitle
\vspace{-6pt}

\begin{abstract}
    \input{abstract}
\end{abstract}

\vspace{6pt}
{
    \hypersetup{linkcolor=black}
    \tableofcontents
}

\vspace{1em}
\section{Introduction}

\input{intro}

\section{Background}

\input{background}

\section{The Steinberg stack}
\label{section-steinberg}
\input{steinberg}

\section{The Mackey formula}
\label{section-mackey}
\input{mackey}

\section{Examples}
\label{section-examples}
\input{quotient}

\phantomsection
\addcontentsline{toc}{section}{References}
\sloppy
\setstretch{1.1}
\renewcommand*{\bibfont}{\normalfont\small}
\printbibliography

\authorinforule

\authorinfo{Chenjing Bu}
    {bucj@mailbox.org}
    {Mathematical Institute, University of Oxford, Oxford OX2 6GG, United Kingdom}

\authorinfo{Lucien Hennecart}
    {lucien.hennecart@u-picardie.fr}
    {Laboratoire Ami\'enois de Math\'ematique Fondamentale et Appliqu\'ee, CNRS UMR 7352, Universit\'e de Picardie Jules Verne, 33 rue Saint Leu, 80000 Amiens, France}
\end{document}

%% file: _preamble.tex
\usepackage[margin=25mm]{geometry}

\usepackage[utf8]{inputenc}
\usepackage[T1]{fontenc}
\usepackage[british]{babel}
\usepackage[style=british]{csquotes}
\usepackage[en-GB]{datetime2}
\DTMlangsetup[en-GB]{ord=omit}
\usepackage{xcolor}

\usepackage{amsthm, thmtools}
\usepackage{mathtools}
\usepackage{titlesec} 
\usepackage{hyperref}
\hypersetup{
    bookmarksnumbered,
    colorlinks,
    linkcolor={cyan!50!blue},
    citecolor={cyan!50!blue},
    urlcolor={cyan!50!blue}
}
\usepackage[capitalise,nameinlink]{cleveref}

\usepackage{tikz}
\usepackage{tikz-cd}

\usepackage[expansion=false,nopatch=footnote]{microtype}
\usepackage[all]{nowidow}

\usepackage[lining,tabular,scale=.9]{FiraSans}
\usepackage[tt=false,semibold]{libertinus}
\usepackage[libertine,smallerops]{newtxmath}
\usepackage[lining,scaled=.8]{FiraMono}
\usepackage[cal=boondox,frak=euler]{mathalpha}
\providecommand{\libertinusDisplay}{\LibertinusDisplay}

\DeclareFontEncoding{MDA}{}{}
\DeclareSymbolFont{mathdesignA}{MDA}{mdput}{m}{n}
\SetSymbolFont{mathdesignA}{bold}{MDA}{mdput}{b}{n}
\DeclareSymbolFontAlphabet{\mathbb}{mathdesignA}

\tikzcdset{
    arrow style=tikz,
    arrows={/tikz/line width=.5pt},
    diagrams={>={Straight Barb[scale=0.8]}}
}

\DeclareFontFamily{OMX}{MnSymbolE}{}
\DeclareSymbolFont{MnLargeSymbols}{OMX}{MnSymbolE}{m}{n}
\SetSymbolFont{MnLargeSymbols}{bold}{OMX}{MnSymbolE}{b}{n}
\DeclareFontShape{OMX}{MnSymbolE}{m}{n}{
    <-6>  MnSymbolE5
   <6-7>  MnSymbolE6
   <7-8>  MnSymbolE7
   <8-9>  MnSymbolE8
   <9-10> MnSymbolE9
  <10-12> MnSymbolE10
  <12->   MnSymbolE12
}{}
\DeclareFontShape{OMX}{MnSymbolE}{b}{n}{
    <-6>  MnSymbolE-Bold5
   <6-7>  MnSymbolE-Bold6
   <7-8>  MnSymbolE-Bold7
   <8-9>  MnSymbolE-Bold8
   <9-10> MnSymbolE-Bold9
  <10-12> MnSymbolE-Bold10
  <12->   MnSymbolE-Bold12
}{}
\DeclareMathDelimiter{[}{\mathopen}{MnLargeSymbols}{'000}{MnLargeSymbols}{'000}
\DeclareMathDelimiter{]}{\mathclose}{MnLargeSymbols}{'005}{MnLargeSymbols}{'005}
\DeclareMathDelimiter{\llbr}{\mathopen}{MnLargeSymbols}{'102}{MnLargeSymbols}{'102}
\DeclareMathDelimiter{\rrbr}{\mathclose}{MnLargeSymbols}{'107}{MnLargeSymbols}{'107}

\usepackage{bm}
\usepackage{cases}
\usepackage{accents}

\usepackage{setspace}
\usepackage{needspace}

\usepackage{subdepth} 

\newcommand{\initlengths}{%
    \setlength{\abovedisplayshortskip}{3pt plus 9pt minus 3pt}%
    \setlength{\belowdisplayshortskip}{9pt plus 9pt minus 9pt}%
    \setlength{\abovedisplayskip}{9pt plus 9pt minus 9pt}%
    \setlength{\belowdisplayskip}{9pt plus 9pt minus 9pt}%
    \hfuzz 1pt%
    \tolerance 400
}

\newcommand{\authorinforule}{\noindent\rule{0.38\textwidth}{0.4pt}}

\newlength{\authorwidth}
\newcommand{\authorinfo}[3]{{%
    \raggedright
    \setlength{\leftskip}{1.5em}
    \setlength{\parindent}{0em}
    \setstretch{1}
    \par%
    {\small%
    \makebox[\authorwidth][l]{#1}%
    \texttt{#2}%
    \\
    #3.}
    \vspace{6pt}\par
}}

\usepackage{titling}
\pretitle{\vspace{-24pt}\setstretch{1.05}\begin{center}\LARGE\libertinusDisplay\fontdimen2\font=0.3em} 
\posttitle{\end{center}\vspace{-6pt}}

\newcommand{\parasep}{9pt plus 3pt minus 3pt}
\titleformat{\section}{\Large\libertinusDisplay\setstretch{1.1}}{\thesection}{1em}{}
\titleformat{\subsection}{\large\firaproportional}{\thesubsection}{1em}{}
\titleformat{\subsubsection}{\firamedium\boldmath}{\thesubsubsection}{1em}{}

\usepackage{tocloft}
\renewenvironment{abstract}{%
    \centering\begin{minipage}{.85\textwidth}%
    \setlength{\parindent}{1.5em}%
    \centerline{\firamedium\abstractname}%
    \par\vspace{12pt}%
}{\end{minipage}\par\vspace{3pt}}

\makeatletter
\declaretheoremstyle[
    spaceabove=\parasep, spacebelow=\parasep,
    postheadspace=.5em,
    headfont=\firamedium\firaproportional,
    headpunct={},
    headformat={\NAME\ \NUMBER\NOTE.},
    notefont=\normalfont,
    notebraces={\ (}{)},
    bodyfont=\itshape,
]{theorem}
\declaretheoremstyle[
    spaceabove=\parasep, spacebelow=\parasep,
    postheadspace=.5em,
    headfont=\firamedium\firaproportional,
    headpunct={},
    headformat={\NAME\NOTE.},
    notefont=\normalfont,
    notebraces={\ (}{)},
    bodyfont=\itshape,
]{theorem*}
\declaretheoremstyle[
    spaceabove=\parasep, spacebelow=\parasep,
    postheadspace=.5em,
    headfont=\firamedium\firaproportional,
    headpunct={},
    headformat={\NAME\ \NUMBER\NOTE.},
    notefont=\normalfont,
    notebraces={\ (}{)},
]{definition}
\declaretheoremstyle[
    spaceabove=\parasep, spacebelow=\parasep,
    postheadspace=.5em,
    headfont=\firamedium\firaproportional,
    headpunct={},
    headformat={\NAME\NOTE.},
    notefont=\normalfont,
    notebraces={\ (}{)},
]{definition*}
\declaretheoremstyle[
    spaceabove=\parasep, spacebelow=\parasep,
    postheadspace=.5em,
    headfont=\firamedium\firaproportional,
    headpunct={},
    headformat={\NOTE},
    notefont=\normalfont\bfseries\boldmath,
    notebraces={\hspace{-.33em}}{.},
]{para}

\renewenvironment{proof}[1][\proofname]{\par
    \pushQED{\qed}%
    \normalfont\trivlist
    \item[\hskip\labelsep\firamedium #1\@addpunct{.}]\ignorespaces
}{%
    \popQED\endtrivlist\@endpefalse
}

\expandafter\def\csname equation*@qed\endcsname{\equation@qed}
\makeatother

\declaretheorem[parent=section, style=theorem, name=Theorem]{theorem}

\declaretheorem[sibling=theorem, style=theorem, name=Lemma]{lemma}
\declaretheorem[sibling=theorem, style=theorem, name=Corollary]{corollary}
\declaretheorem[sibling=theorem, style=theorem, name=Proposition]{proposition}

\declaretheorem[numbered=no, style=theorem*, name=Theorem]{theorem*}
\declaretheorem[numbered=no, style=theorem*, name=Lemma]{lemma*}
\declaretheorem[numbered=no, style=theorem*, name=Proposition]{proposition*}
\declaretheorem[sibling=theorem, style=definition, name=Definition]{definition}
\declaretheorem[sibling=theorem, style=definition, name=Remark]{remark}
\declaretheorem[sibling=theorem, style=definition, name=Example]{example}
\declaretheorem[sibling=theorem, style=definition, name=Assumption]{assumption}

\numberwithin{equation}{section}

\crefdefaultlabelformat{#2{\upshape#1}#3}

\crefformat{equation}{#2{\upshape(#1)}#3}
\crefrangeformat{equation}{{\upshape#3(#1)#4--#5(#2)#6}}
\crefmultiformat{equation}{#2{\upshape(#1)}#3}{ and #2{\upshape(#1)}#3}{, #2{\upshape(#1)}#3}{, and~#2{\upshape(#1)}#3}

\crefformat{enumi}{#2{\upshape#1}#3}
\crefrangeformat{enumi}{{\upshape#3#1#4--#5#2#6}}
\crefmultiformat{enumi}{#2{\upshape#1}#3}{ and #2{\upshape#1}#3}{, #2{\upshape#1}#3}{, and~#2{\upshape#1}#3}

\crefformat{section}{#2{\upshape\S#1}#3}
\crefrangeformat{section}{{\upshape#3\S\S#1#4--#5#2#6}}
\crefmultiformat{section}{{\upshape#2\S\S#1#3}}{ and {\upshape#2#1#3}}{, {\upshape#2#1#3}}{, and~{\upshape#2#1#3}}
\crefformat{subsection}{#2{\upshape\S#1}#3}
\crefrangeformat{subsection}{{\upshape#3\S\S#1#4--#5#2#6}}
\crefmultiformat{subsection}{{\upshape#2\S\S#1#3}}{ and {\upshape#2#1#3}}{, {\upshape#2#1#3}}{, and~{\upshape#2#1#3}}
\crefformat{para}{#2{\upshape\S#1}#3}
\crefrangeformat{para}{{\upshape#3\S\S#1#4--#5#2#6}}
\crefmultiformat{para}{{\upshape#2\S\S#1#3}}{ and {\upshape#2#1#3}}{, {\upshape#2#1#3}}{, and~{\upshape#2#1#3}}

\crefname{figure}{Figure}{Figures}

\usepackage{enumitem}
\setlist{noitemsep}
\setlist[enumerate]{label=\textnormal{(\roman*)}}

\usepackage[labelsep=period, labelfont={bf}]{caption}

\usepackage[bottom, hang, multiple]{footmisc}
\renewcommand{\footnoterule}{%
  \kern3pt
  \hrule width 0.38\columnwidth height 0.5pt
  \kern-3.5pt}

\usepackage{etoolbox}
\makeatletter
\patchcmd{\@makefntext}{\ifFN@hangfoot\bgroup}%
{\ifFN@hangfoot\bgroup\def\@makefnmark{\rlap{\firamedium\firaproportional\@thefnmark}}}{}{}%
\makeatother

\usepackage[
    backend=biber,
    style=oxnum,
    giveninits,
    maxbibnames=10,
    maxcitenames=5,
    sorting=nyt
]{biblatex}

\DeclareFieldFormat*{title}{\textit{#1}}
\DeclareFieldFormat*{journaltitle}{#1}
\renewbibmacro*{volume+number+eid}{%
    \printfield{volume}%
    \setunit*{.}%
    \printfield{number}%
}%
\DeclareFieldFormat[article]{version}{\IfDecimal{#1}{version~}{}#1}
\DefineBibliographyStrings{english}{
  withafterword = {with an appendix by},
}

\DeclareFieldFormat{pages}{%
  \iffieldundef{bookpagination}%
    {#1}%
    {\mkpageprefix[bookpagination]{#1}}%
}

%% file: _commands.tex
\newcommand{\calBun}{{\mathcal{B}\mkern-2mu\mathit{un}\mkern1mu}}
\newcommand{\calFilt}{{\mathcal{F}\mkern-3mu\mathit{ilt}\mkern1mu}}
\newcommand{\calGrad}{{\mathcal{G}\mkern-3mu\mathit{rad}\mkern1mu}}
\newcommand{\calHom}{{\mathcal{H}\mkern-4mu\mathit{om}\mkern1mu}}
\newcommand{\calMap}{{\mathcal{M}\mkern-3mu\mathit{ap}\mkern2mu}}
\newcommand{\calSt}{{\mathcal{S}\mkern-2mu\mathit{t}\mkern1mu}}
\newcommand{\git}{{/\mspace{-5mu}/}}
\newcommand{\longhookrightarrow}{\lhook\joinrel\longrightarrow}
\newcommand{\longsimto}{\mathrel{\overset{\smash{\raisebox{-.8ex}{$\sim$}}\mspace{3mu}}{\longrightarrow}}}

\renewcommand{\geq}{\geqslant}
\renewcommand{\leq}{\leqslant}

\makeatletter
\def\big#1{{\hbox{$\left#1\vbox to10\p@{}\right.\n@space$}}}
\makeatother

%% file: abstract.tex
We study parabolic induction and restriction functors
on derived categories of constructible sheaves, D-modules,
and other sheaf theories on smooth stacks with good or adequate moduli spaces.
We prove a commutation rule between them in the form of a Mackey filtration,
indexed by the set of relative positions of two elements
in the component lattice,
a combinatorial structure that characterizes
inductions and restrictions on a given stack.
This generalizes similar results for adjoint quotient stacks
studied in generalized Springer theory
and the theory of character sheaves,
and for stacks of quiver representations
appearing in Lusztig's categorification of quantum groups.
We also apply our result to the stack of semistable principal bundles on a curve,
as a first step towards Mackey formulas in geometric Langlands correspondences.

%% file: intro.tex
Induction and restriction are among the basic operations
in representation theory in various contexts,
including representations of finite groups,
character sheaves on reductive groups and Lie algebras,
various types of Hall algebras and quantum groups,
categorical Langlands correspondences, etc.
Their interaction is often governed by a \emph{Mackey formula} of the form
\begin{equation*}
    \mathrm{Res} \circ \mathrm{Ind} \sim
    \bigoplus_{i \in I} \mathrm{Ind}_i \circ T_i \circ \mathrm{Res}_i \ ,
\end{equation*}
where $I$ denotes a set of \emph{relative positions}, $\mathrm{Ind}$, $\mathrm{Ind}_i$,
$\mathrm{Res}$, and $\mathrm{Res}_i$
denote induction and restriction functors;
`$\sim$' sometimes means an isomorphism,
and sometimes means a filtration with the associated graded pieces being the right-hand side;
and $T_i$ is a twist functor, which can be trivial in some contexts.

In this paper, we focus on parabolic induction and restriction
of sheaves on algebraic stacks.
This includes many of the contexts mentioned above,
such as character sheaves,
Lusztig's categorification of quantum groups,
and many versions of Langlands correspondences.

Recent developments in the theory of algebraic stacks
make it possible to formulate these operations intrinsically to the stack.
Namely, given a stack $\mathcal{X}$ satisfying suitable conditions,
all possible parabolic inductions and restrictions that can be done on the stack
are parametrized by a combinatorial structure called the
\emph{component lattice} $\mathrm{CL} (\mathcal{X})$,
following \cite{bu2025intrinsic},
building upon the work of \textcite{halpern-leistner-instability}.
This generalizes the cocharacter lattice of an algebraic group,
where each cocharacter determines a parabolic subgroup to do induction.
For each cocharacter $\lambda \in \mathrm{CL} (\mathcal{X})$,
we have a correspondence
\[
    \mathcal{X}_\lambda
    \overset{q_\lambda}{\longleftarrow}
    \mathcal{X}_\lambda^+
    \overset{p_\lambda}{\longrightarrow}
    \mathcal{X} \ ,
\]
where $\mathcal{X}_\lambda^+$ plays the role of the parabolic subgroup,
and~$\mathcal{X}_\lambda$ plays the role of the Levi quotient.
See \cref{subsec-component-lattice}
for the precise definitions.
We then define the induction and restriction functors
\[
    \mathrm{Ind}_\lambda =
    (p_\lambda)_! \, q_\lambda^* \ ,
    \qquad
    \mathrm{Res}_\lambda =
    (q_\lambda)_! \, p_\lambda^*
\]
and also their right-adjoints on a suitable sheaf theory, see \cref{def-ind-res}.

Our main result is the following:

\begin{theorem*}[\cref{theorem:mackey-formula}]
    Let~$\mathcal{X}$ be a smooth algebraic stack
    over an algebraic space~$S$
    satisfying \cref{assumption-stratification},
    and let $\lambda, \mu \in \mathrm{CL} (\mathcal{X})$.

    Then, for a sheaf theory satisfying \cref{assumption-six-functors},
    the composition
    $\mathrm{Res}_{\mu} \circ \mathrm{Ind}_{\lambda}$
    admits a finite filtration with associated graded pieces
    \begin{align*}
        \mathrm{gr}_i (\mathrm{Res}_{\mu} \circ \mathrm{Ind}_{\lambda})
        & \simeq
        \mathrm{Ind}_{\lambda_i} \circ T^i_{\lambda, \mu} \circ \mathrm{Res}_{\mu_i} \ ,
    \end{align*}
    where~$i$ runs through the ordered set of all relative positions of~$\lambda$ and~$\mu$
    as in \cref{def-relative-position},
    $\lambda_i \in \mathrm{CL} (\mathcal{X}_\mu)$
    and~$\mu_i \in \mathrm{CL} (\mathcal{X}_\lambda)$
    are lifts of~$\lambda$ and~$\mu$ corresponding to~$i$,
    and $T^i_{\lambda, \mu}$ is a Tate twist as in \cref{equation:twist}.
\end{theorem*}

\noindent
Here, the assumptions on~$\mathcal{X}$ include the existence of
a good or adequate moduli space in the sense of
\textcite{alper-2013-good,alper-2014-adequate}.
The sheaf theory can be taken to be
constructible sheaves,
mixed Hodge modules,
$\ell$-adic sheaves,
or D-modules,
as we will discuss in \cref{subsec-six-functors}.

The main idea of the proof is to construct a stratification of
the \emph{Steinberg stack} of~$\mathcal{X}$ (\cref{section-steinberg}),
defined as a fibre product
$\calSt (\mathcal{X})_{\lambda, \mu} =
\mathcal{X}_\lambda^+ \times_{\mathcal{X}} \mathcal{X}_\mu^+$,
where each stratum corresponds to one term in the Mackey filtration.
This stratification is given in \cref{theorem:stratification},
and generalizes the Bruhat decomposition of an algebraic group.
A cohomological Mackey formula for linear quotient stacks
was obtained using a similar approach by the second-named author
\cite{hennecart-mackey}.

In examples (see \cref{section-examples}), our result recovers known instances of Mackey formulas
for constructible sheaves or D-modules.
The Mackey formula is crucial in the study of
cuspidal character sheaves or D-modules on reductive Lie algebras,
as in \cite{gunningham2018generalized};
see also \cite{mirkovic2004character,zhou2024character} for its characteristic~$p$ analogue.
In this situation, the set of relative positions is indexed by double cosets of the Weyl group.
In the context of quivers,
we recover the bialgebra compatibility between
multiplication and comultiplication for quantum groups
in the categorification of Lusztig \cite{lusztig1991quivers,lusztig1993introduction}.
We explain in \Cref{example:quiver-representations}
how to derive some structural results of Lusztig from our study.

Character sheaves on reductive groups were introduced by Lusztig in the 1980s
\cite{
    lusztig-1985-character-i,
    lusztig-1985-character-ii,
    lusztig-1985-character-iii,
    lusztig-1986-character-iv,
    lusztig-1986-character-v,
    lusztig2006character},
to give a geometrization of the representation theory of finite groups of Lie type.
The idea was to find perverse sheaves or constructible complexes on the group
whose Frobenius traces recover, or are closely related to,
the irreducible characters of the group.
The theory of character sheaves provides such a class of perverse sheaves.
Their Frobenius traces give \emph{almost characters},
which do not exactly coincide with irreducible characters of the group
but are better behaved from the point of view of geometry.
Then, induction and restriction of character sheaves
correspond to induction and restriction of
representations of groups via the character or Frobenius trace.
The historical motivation comes from Deligne--Lusztig theory~\cite{deligne1976representations}.

Character sheaves on Lie algebras form the linear analogue of the group situation,
and over finite fields,
they provide natural bases of the space of invariant functions \cite{letellier2004fourier}.
On Lie algebras, a very interesting structure is provided by the Fourier transform,
which relates invariant functions on a Lie algebra to its dual.
The theory of character sheaves forms a geometric version of harmonic analysis.

In Lusztig's theory, unipotent character sheaves on a connected
reductive group arise as simple perverse constituents of geometric
complexes associated with Weyl group elements and the trivial local
system on a maximal torus
\cite{lusztig-1985-character-i,lusztig-1986-character-v}.
The principal-series objects are obtained by induction from the
constant sheaf on a maximal torus, but general unipotent character
sheaves also include cuspidal objects \cite{MarsSpringer1989}.
Related induction constructions underlie Lusztig's perverse sheaves
on quiver representation stacks \cite{lusztig1991quivers} and
Schiffmann's spherical Eisenstein sheaves on stacks of coherent sheaves
\cite{Schiffmann2006,Schiffmann2012,Hennecart2024}.

The formalism of component lattices \cite{bu2025intrinsic,halpern-leistner-instability}
provides intrinsic analogues of parabolic induction for more general algebraic stacks,
as mentioned earlier.
This motivates studying, for example, for smooth algebraic stacks under suitable hypotheses,
the simple perverse constituents
of complexes obtained as parabolic inductions of constant sheaves.
A natural problem is to describe these sheaves explicitly,
in particular their supports and the irreducible local systems whose
intermediate extensions give rise to them.

We now explain our result in the context of
geometric Langlands correspondences
\cite{glc-i,glc-ii,glc-iii,glc-iv,glc-v},
where induction and restriction functors are also known as
Eisenstein and constant term functors
\cite{BravermanGaitsgory2002,drinfeld2016geometric}.
Our Mackey formula applies to the moduli stack
$\calBun_G^{\mathrm{ss}} (C)$
of semistable $G$-bundles on a geometrically connected smooth projective curve~$C$
over a field~$k$ with~$G$ split reductive (see \cref{subsec-bunss}),
but it does not apply to the full moduli stack
$\calBun_G (C)$
of all $G$-bundles on~$C$,
since we assume that the stack is quasi-compact and
has a good or adequate moduli space.
In particular, it does not apply to the full Eisenstein and constant term functors
on the automorphic side.
We expect a more involved version of the Mackey formula to hold in that case,
possibly with modified induction and restriction functors
and more sophisticated twist functors,
which we hope to explore in future work.
See also \textcite[Theorem~3.6.3]{hansen-mann-langlands}
for a similar formula on the spectral side of the
(arithmetic) local Langlands correspondence.

We also hope to explore Mackey formulas for coherent sheaves
on smooth and quasi-smooth stacks,
which the current paper does not cover,
possibly using tools from \cite{bu-padurariu-toda}.
These categories appear on the spectral side of Langlands correspondences,
with one version already appearing in \cite{hansen-mann-langlands} as mentioned above.
They are also natural categorifications of $K$-theoretic Hall algebras,
which give quantum loop algebras when applied to quiver representations,
and a Mackey formula would give a categorified bialgebra compatibility there.

Another interesting direction is to construct \emph{perverse schobers}
based on our Mackey formulas.
Perverse schobers are categorifications of
perverse sheaves on hyperplane arrangements,
proposed by \textcite{kapranov-schechtman-schobers}.
Their general definition is still unknown,
except in the case of type~A arrangements discussed by \textcite{dyckerhoff-wedrich}.
Heuristically, a perverse schober can be thought of as
a collection of categories
with induction and restriction functors between them
satisfying Mackey formulas with higher homotopy coherences.
This is made precise in type~A in \cite[Theorem~3.23]{dyckerhoff-wedrich},
and a decategorification of this idea for general Coxeter types is discussed in
\cite{kapranov-schechtman-2022-induction,kapranov-schechtman-schiffmann-yuan}.
We expect that for a general stack satisfying our assumptions,
any suitable sheaf theory will define a suitable notion of a perverse schober
on the (complexified) component lattice of the stack,
equipped with a canonical hyperplane arrangement,
and we hope the geometric setup of this paper will provide a way
to study the higher coherences that will be involved.

Finally, the study of perverse sheaves on smooth algebraic stacks
plays an important role in geometric representation theory,
particularly in Springer theory and its generalizations.
A fundamental example is the geometric realization of quiver Hecke
algebras (also called Khovanov--Lauda--Rouquier algebras)
\cite{KhovanovLauda2009,Rouquier2012}
as Ext-algebras of Lusztig's complexes on stacks of quiver
representations \cite{lusztig1991quivers,VaragnoloVasserot2011}.
These constructions have been generalized to representations
of reductive groups \cite{Sauter2013}, with further formality
results established in \cite{Antor2024}.
More generally, given a smooth algebraic stack $\mathcal{X}$
equipped with suitable induction correspondences,
a natural problem is to compute the graded Ext-algebra
\[
\operatorname{Ext}^\bullet_{\mathcal{X}}
\biggl(
\bigoplus_{\alpha}\operatorname{Ind}_{\alpha}(\mathbb Q) \ , \
\bigoplus_{\alpha}\operatorname{Ind}_{\alpha}(\mathbb Q)
\biggr)
\]
of complexes obtained by inducing suitably shifted constant sheaves.
Understanding these Ext-algebras, and in particular finding explicit
presentations by generators and relations, is an interesting problem
which may uncover new algebraic structures and provide geometric
descriptions of the corresponding derived categories.
\subsection*{Acknowledgements}
Chenjing Bu would like to thank
Qiyuan Chen,
Xinyu Li,
Tudor Pădurariu,
Peng Shan,
and Jiangfan Yuan
for helpful discussions.
Chenjing Bu would also like to thank the
Laboratoire Amiénois de Mathématique Fondamentale et Appliquée (LAMFA) in Amiens
for hosting him during his visit in November~2025,
during which a part of this work was realized,
and the authors would like to thank the CNRS for providing the funding
that allowed this visit to take place.
Chenjing Bu was supported by EPSRC grant reference EP/X040674/1.

\subsection*{AI declaration}

This paper was written entirely by the authors.
Generative~AI was used to discuss ideas
and to proofread the paper.

%% file: background.tex
\subsection{The component lattice}
\label{subsec-component-lattice}

We recall the notion of the \emph{component lattice} of a stack,
following \textcite{bu2025intrinsic}.
Roughly speaking, it is a generalization of
the cocharacter lattice of an algebraic group,
where we also take the quotient by the Weyl group.
For a general stack, it can be thought of as
glued from cocharacter lattices of the stabilizer groups.

Throughout this subsection,
we let~$S$ denote a locally noetherian algebraic space,
and~$\mathcal{X}$ an algebraic stack locally of finite type over~$S$
with affine stabilizers and separated diagonal.

\begin{definition}[Graded and filtered points]
    \label{def-grad-filt}
    Following \textcite{halpern-leistner-instability},
    we define the stacks of \emph{graded points} and \emph{filtered points}
    of~$\mathcal{X}$ as the mapping stacks
    \begin{equation}
        \label{eq-def-grad-filt}
        \calGrad (\mathcal{X}) =
        \calMap (\mathrm{B} \mathbb{G}_{\mathrm{m}}, \mathcal{X}) \ ,
        \qquad
        \calFilt (\mathcal{X}) =
        \calMap (\mathbb{A}^1 / \mathbb{G}_{\mathrm{m}}, \mathcal{X}) \ ,
    \end{equation}
    where~$\mathbb{G}_\mathrm{m}$ acts on~$\mathbb{A}^1$ by scaling,
    and all stacks are over~$S$.
    These are again algebraic stacks locally of finite type over~$S$
    with affine stabilizers and separated diagonal.

    There are the induced morphisms
    \begin{equation}
        \label{eq-grad-filt-maps}
        \begin{tikzcd}[column sep=2em]
            \calGrad (\mathcal{X})
            \ar[r, shift right, "s"']
            \ar[rr, start anchor=north east, end anchor=north west, bend left=20, "u"]
            &
            \calFilt (\mathcal{X})
            \ar[l, shift right, "q"']
            \ar[r, "p"]
            &
            \mathcal{X} \ ,
        \end{tikzcd}
    \end{equation}
    where $p, q, s, u$ are induced by the open immersion
    $1 \colon S \to \mathbb{A}^1 / \mathbb{G}_{\mathrm{m}}$,
    the closed immersion
    $0 \colon \mathrm{B} \mathbb{G}_{\mathrm{m}} \to \mathbb{A}^1 / \mathbb{G}_{\mathrm{m}}$,
    the projection
    $\mathbb{A}^1 / \mathbb{G}_{\mathrm{m}} \to \mathrm{B} \mathbb{G}_{\mathrm{m}}$,
    and the morphism
    $S \to \mathrm{B} \mathbb{G}_{\mathrm{m}}$,
    respectively.
    The morphism~$q$ is an $\mathbb{A}^1$-deformation retract
    with section~$s$.

    A typical example is when $S = \operatorname{Spec} k$ for a field~$k$,
    and $\mathcal{X} = U / G$ for a $k$-scheme~$U$
    and a smooth affine algebraic group~$G$ over~$k$
    with a split maximal torus $T \subset G$.
    In this case, as in \cite[Theorem~1.4.8]{halpern-leistner-instability}, we have
    \begin{equation*}
        \calGrad (U / G) \simeq
        \coprod_{\lambda \in \Lambda_T / W}
        U_\lambda / L_\lambda \ ,
        \qquad
        \calFilt (U / G) \simeq
        \coprod_{\lambda \in \Lambda_T / W}
        U_\lambda^+ / P_\lambda \ ,
    \end{equation*}
    where $\Lambda_T = \mathrm{Hom} (\mathbb{G}_\mathrm{m}, T)$
    is the cocharacter lattice,
    $W = \mathrm{N}_G (T) / \mathrm{Z}_G (T)$ is the Weyl group,
    $L_\lambda, P_\lambda$ are the Levi and parabolic subgroups associated to~$\lambda$,
    and $U_\lambda, U_\lambda^+$
    are the fixed and attracting loci, respectively.
    See \emph{loc.~cit.} and \cite[\S 1]{drinfeld-gaitsgory-2014-braden} for details.

    Another class of examples is when $S = \operatorname{Spec} k$ for a field~$k$,
    and $\mathcal{X}$ is the moduli stack of objects in a $k$-linear abelian category.
    In this case, $\calGrad (\mathcal{X})$ and $\calFilt (\mathcal{X})$
    can often be identified with the moduli stacks of
    $\mathbb{Z}$-graded objects and $\mathbb{Z}$-indexed filtrations, respectively,
    in the abelian category.
\end{definition}

\begin{definition}[Quasi-compact graded points]
    \label{def-qc-graded-points}
    Following \cite[\S 6.2.2]{bu2025intrinsic},
    we introduce a mild condition on the stack~$\mathcal{X}$
    that will be used frequently below.

    Writing
    $\calGrad^n (\mathcal{X}) =
    \calMap (\mathrm{B} \mathbb{G}_{\mathrm{m}}^n, \mathcal{X})$
    for $n \geq 0$,
    we say that~$\mathcal{X}$ has \emph{quasi-compact graded points},
    if the forgetful morphism
    $\calGrad^n (\mathcal{X}) \to \mathcal{X}$
    is quasi-compact on every connected component of
    $\calGrad^n (\mathcal{X})$ for all~$n$.
    This is a very mild condition and is usually satisfied in practice.

    When $\mathcal{X}$ is quasi-compact and has quasi-compact graded points,
    all connected components of
    $\calGrad (\mathcal{X})$ and $\calFilt (\mathcal{X})$
    are quasi-compact,
    by \textcite[Lemma~1.3.8]{halpern-leistner-instability}.
\end{definition}

\begin{definition}[Component lattice]
    \label{def-cl}
    Now following \textcite{bu2025intrinsic},
    we define the \emph{component lattice} of~$\mathcal{X}$ as the functor
    \begin{align*}
        \mathrm{CL} (\mathcal{X}) \colon
        (\mathbb{Z} \textnormal{-} \mathsf{Mod}^\mathrm{ff})^\mathrm{op}
        & \longrightarrow
        \mathsf{Set} \ ,
        \\
        \Lambda
        & \longmapsto
        \uppi_0 ( \calGrad^\Lambda (\mathcal{X}) ) \ ,
    \end{align*}
    where $\mathbb{Z} \textnormal{-} \mathsf{Mod}^\mathrm{ff}$
    is the category of finitely generated free $\mathbb{Z}$-modules,
    and we write
    $\calGrad^\Lambda (\mathcal{X}) = \calMap (\mathrm{B} T_\Lambda, \mathcal{X})$
    for the split torus $T_\Lambda$ over~$S$
    whose cocharacter lattice is~$\Lambda$.

    Such a functor is called a \emph{formal lattice}.
    We see a formal lattice as glued from objects
    $\mathbb{Z}^n \in \mathbb{Z} \textnormal{-} \mathsf{Mod}^\mathrm{ff}$
    along morphisms in $\mathbb{Z} \textnormal{-} \mathsf{Mod}^\mathrm{ff}$.
    In particular, we define the underlying set
    \begin{equation*}
        | \mathrm{CL} (\mathcal{X}) | =
        \mathrm{CL} (\mathcal{X}) (\mathbb{Z}) \simeq
        \uppi_0 ( \calGrad (\mathcal{X}) ) \ .
    \end{equation*}
    When~$\mathcal{X}$ is quasi-compact and has quasi-compact graded points
    (\cref{def-qc-graded-points}),
    $\mathrm{CL} (\mathcal{X})$ can be glued from
    finitely many lattices $\mathbb{Z}^n$
    along finitely many maps, by \cite[Theorem~6.2.3]{bu2025intrinsic}.

    For example, when $\mathcal{X} = U / G$ as in \cref{def-grad-filt},
    if each $U_\lambda$ is connected
    (such as when~$U$ is a $G$-representation), then we have
    $\mathrm{CL} (U / G) \simeq \Lambda_T / W$
    as a formal lattice.

    For each $\lambda \in \mathrm{CL} (\mathcal{X})$
    (meaning $\lambda \in | \mathrm{CL} (\mathcal{X}) |$),
    we denote by
    \begin{equation*}
        \mathcal{X}_\lambda \subset \calGrad (\mathcal{X}) \ ,
        \qquad
        \mathcal{X}_\lambda^+ \subset \calFilt (\mathcal{X})
    \end{equation*}
    the connected components corresponding to~$\lambda$,
    noting that the maps $q, s$ in
    \cref{eq-grad-filt-maps}
    induce inverse isomorphisms on $\uppi_0$.
    We denote the restriction of the maps in
    \cref{eq-grad-filt-maps}
    to these components by
    \begin{equation*}
        \begin{tikzcd}[column sep=2em]
            \mathcal{X}_\lambda
            \ar[r, shift right, "s_\lambda"']
            \ar[rr, start anchor=north east, end anchor=north west, bend left=20, "u_\lambda"]
            &
            \mathcal{X}_\lambda^+
            \ar[l, shift right, "q_\lambda"']
            \ar[r, "p_\lambda"]
            &
            \mathcal{X} \ .
        \end{tikzcd}
    \end{equation*}

    We also use the notations $\mathcal{X}_\lambda$ and $\mathcal{X}_\lambda^+$
    in the following broader sense:
    Given a morphism $\mathcal{X} \to \mathcal{Y}$
    and $\lambda \in \mathrm{CL} (\mathcal{Y})$,
    we define $\mathcal{X}_\lambda \subset \calGrad (\mathcal{X})$
    and $\mathcal{X}_\lambda^+ \subset \calFilt (\mathcal{X})$
    as preimages of $\mathcal{Y}_\lambda$ and~$\mathcal{Y}_\lambda^+$
    under the induced morphisms,
    so they will be (possibly empty) disjoint unions of
    $\mathcal{X}_{\lambda'}$ and $\mathcal{X}_{\lambda'}^+$
    for $\lambda' \in \mathrm{CL} (\mathcal{X})$
    in the preimage of~$\lambda$.
    In particular, when $\mathcal{X} = U / G$ for an algebraic space~$U$
    and a group scheme~$G$ over a connected base,
    we may take $\mathcal{Y} = \mathrm{B} G$ and write
    $\mathcal{X}_\lambda$, $\mathcal{X}_\lambda^+$
    for cocharacters $\lambda \colon \mathbb{G}_\mathrm{m} \to G$,
    identifying~$\lambda$ with its induced element in $\mathrm{CL} (\mathrm{B} G)$.
\end{definition}

\begin{definition}[Faces and cones]
    \label{def-faces-cones}
    An $n$-dimensional \emph{face} of~$\mathrm{CL} (\mathcal{X})$
    is a map of formal lattices
    $\alpha \colon \mathbb{Z}^n \to \mathrm{CL} (\mathcal{X})$,
    or equivalently, an element
    $\alpha \in \mathrm{CL} (\mathcal{X}) (\mathbb{Z}^n)$.
    For a polyhedral cone
    $\Sigma \subset \mathbb{Z}^n$
    which is the intersection of an $n$-dimensional convex polyhedral cone in~$\mathbb{Q}^n$
    with~$\mathbb{Z}^n$,
    we also consider the restriction
    $\sigma = \alpha |_\Sigma \colon \Sigma \to \mathrm{CL} (\mathcal{X})$
    as a \emph{cone} of~$\mathrm{CL} (\mathcal{X})$.
    In this situation, we have a similar diagram
    \begin{equation*}
        \begin{tikzcd}[column sep=2em]
            \mathcal{X}_\alpha
            \ar[r, shift right, "s_\sigma"']
            \ar[rr, start anchor=north east, end anchor=north west, bend left=20, "u_\alpha"]
            &
            \mathcal{X}_\sigma^+
            \ar[l, shift right, "q_\sigma"']
            \ar[r, "p_\sigma"]
            &
            \mathcal{X} \ .
        \end{tikzcd}
    \end{equation*}
    See \cite{bu2025intrinsic} for the precise definitions.

    When $\mathcal{X}$ is quasi-compact and has quasi-compact graded points,
    we have a canonical isomorphism
    $\mathcal{X}_\lambda \simeq \mathcal{X}_\alpha$
    for any generic $\lambda \in \mathrm{CL} (\mathcal{X})$
    in the image of~$\alpha$,
    and we have
    $\mathcal{X}_\lambda^+ \simeq \mathcal{X}_\sigma^+$
    for any generic $\lambda \in \mathrm{CL} (\mathcal{X})$
    in the image of~$\sigma$
    provided that~$\sigma$ is contained in a chamber in
    a certain hyperplane arrangement
    \cite[\S 5.3.2]{bu2025intrinsic}
    in the face spanned by it.
    For a larger cone~$\sigma$,
    one cannot in general identify~$\mathcal{X}_\sigma^+$
    with a component of $\calFilt (\mathcal{X})$.
\end{definition}

\begin{definition}[Hall category]
    \label{def-hall-cat}
    The faces and cones of~$\mathrm{CL} (\mathcal{X})$
    form a combinatorial object called the \emph{Hall category}
    as in \cite[\S\S 6.3.3 and 6.3.5]{bu2025intrinsic}.

    To define this category, we need the notion of \emph{special faces}
    as in \cite[\S 4.1.2]{bu2025intrinsic}.
    Roughly speaking,
    $| \mathrm{CL} (\mathcal{X}) |$ often carries a wall-chamber decomposition
    such that the stacks~$\mathcal{X}_\lambda$ and~$\mathcal{X}_\lambda^+$
    do not change when~$\lambda$ is moved inside a chamber.
    A special face is, in rough terms, a face that is either maximal,
    or an intersection of walls that cut out this decomposition.
    On each special face, there is a canonical hyperplane arrangement
    given by \emph{cotangent weights}
    as in \cite[\S 5.3.2]{bu2025intrinsic}.

    The category $\mathsf{Hall} (\mathcal{X})$ is defined
    when each special face has finitely many cotangent weights.
    It is defined as follows:

    \begin{itemize}
        \item
            Its objects are special faces of~$\mathrm{CL} (\mathcal{X})$.
        \item
            A morphism $\alpha \overset{\sigma}{\to} \beta$
            is an inclusion of the underlying lattice of~$\alpha$ to that of~$\beta$
            compatible with the maps to $\mathrm{CL} (\mathcal{X})$,
            together with a (closed) chamber~$\sigma$
            in the hyperplane arrangement in the underlying lattice of~$\beta$
            consisting of hyperplanes defined by cotangent weights
            that contain the image of~$\alpha$.
        \item
            The composition
            $\alpha \overset{\sigma}{\to} \beta \overset{\tau}{\to} \gamma$
            is given by the cone
            $\alpha \overset{\smash{\sigma \uparrow \tau}}{\to} \gamma$,
            where $\sigma \uparrow \tau$ is the chamber
            uniquely determined by the property that
            $\sigma \subset \sigma \uparrow \tau \subset \tau$.
    \end{itemize}
    The Hall category captures the combinatorics of
    all induction and restriction operations that one can do on~$\mathcal{X}$.
    Induction behaves like a pullback along a morphism in $\mathsf{Hall} (\mathcal{X})$,
    whereas restriction behaves like a pushforward.
    See \cref{prop-ind-res-associativity}
    for the precise meaning of this.
\end{definition}

\noindent
Automorphism groups in the Hall category are often related to Weyl groups.
For example, when
$\mathcal{X} = \mathrm{B} G$
for a smooth affine algebraic group~$G$ with a split maximal torus~$T$
over a field~$k$, the special face
$\Lambda_T \to \Lambda_T / W \simeq \mathrm{CL} (\mathcal{X})$
has automorphism group~$W$.

If $\mathcal{X}$ is quasi-compact and has quasi-compact graded points,
then $\mathsf{Hall} (\mathcal{X})$ is always defined,
and is moreover a finite category,
in that it is equivalent to a category with finitely many objects and morphisms.
In this case,
for any $\alpha \in \mathsf{Hall} (\mathcal{X})$,
we have
$\mathcal{X}_\alpha \simeq \mathcal{X}_\lambda$
for some $\lambda \in \mathrm{CL} (\mathcal{X})$,
and for any morphism
$\alpha \overset{\sigma}{\to} \beta$
in~$\mathsf{Hall} (\mathcal{X})$,
we have
$\mathcal{X}_\sigma^+ \simeq (\mathcal{X}_\alpha)_\lambda^+$
for some $\lambda \in \mathrm{CL} (\mathcal{X}_\alpha)$.

In this case, we also have the following
\emph{associativity theorem},
as in \cite[Theorem~6.3.8]{bu2025intrinsic}:

\begin{theorem}[Associativity]
    \label{thm-associativity}
    If $\mathcal{X}$ is quasi-compact and has quasi-compact graded points,
    the assignment
    $\alpha \mapsto \mathcal{X}_\alpha$,
    $\sigma \mapsto \mathcal{X}_\sigma^+$
    defines a functor from $\mathsf{Hall} (\mathcal{X})$
    to the category of spans of stacks over~$S$.
    In particular, for morphisms
    $\alpha \overset{\sigma}{\to} \beta \overset{\tau}{\to} \gamma$
    in $\mathsf{Hall} (\mathcal{X})$, we have a canonical diagram
    \begin{equation*}
        \begin{tikzcd}[column sep={4em, between origins}, row sep={2.5em, between origins}]
            &&
            \mathcal{X}_{\sigma \uparrow \tau}^+
            \ar[dl] \ar[dr]
            \ar[ddll, bend right=20, shift right, "q_{\sigma \uparrow \tau}"']
            \ar[ddrr, bend left=20, shift left, "p_{\sigma \uparrow \tau}"]
            \ar[dd, phantom, pos=.2, "\ulcorner" rotate=-45]
            &&
            \\
            &
            \mathcal{X}_\tau^+
            \ar[dl, "q_\tau"] \ar[dr, "p_\tau"']
            &&
            \mathcal{X}_\sigma^+
            \ar[dl, "q_\sigma"] \ar[dr, "p_\sigma"']
            \\
            \mathcal{X}_\gamma && \mathcal{X}_\beta && \mathcal{X}_\alpha \rlap{ ,}
        \end{tikzcd}
    \end{equation*}
    where the middle square is cartesian.
\end{theorem}

\begin{definition}[Norms on graded points]
    \label{def-norm}
    Following \textcite[Definition~4.1.12]{halpern-leistner-instability},
    define a (\emph{quadratic}) \emph{norm on graded points} on~$\mathcal{X}$,
    or a (\emph{quadratic}) \emph{norm} on~$\mathrm{CL} (\mathcal{X})$,
    to be a function
    \begin{equation*}
        \nu \colon |\mathrm{CL} (\mathcal{X}) \otimes \mathbb{R}|
        \longrightarrow \mathbb{R}_{\geq 0} \ ,
    \end{equation*}
    where
    $\mathrm{CL} (\mathcal{X}) \otimes \mathbb{R}$
    is defined as in \cite[\S 2.1.8]{bu2025intrinsic},
    such that for any face
    $\alpha \colon \mathbb{R}^n \to \mathrm{CL} (\mathcal{X}) \otimes \mathbb{R}$
    that does not factor through $\mathbb{R}^{n-1}$,
    the pullback $\alpha^* (\nu) = \nu \circ \alpha$ is a
    (quadratic) norm on~$\mathbb{R}^n$.
\end{definition}

\noindent
Quadratic norms on graded points exist on many stacks of interest.
See \cite[Examples~4.1.13 and~4.1.17, \S 6.4]{halpern-leistner-instability}.

The existence of a norm on graded points
implies having quasi-compact graded points in the sense of
\cref{def-qc-graded-points},
by \textcite[Proposition~3.8.2]{halpern-leistner-instability}.

Given a quadratic norm on graded points,
for each geometric point $x \in \mathcal{X} (\bar{k})$
with stabilizer group~$G_x$,
the norm restricts to a Weyl-invariant inner product
on the real cocharacter lattice $\Lambda_T \otimes \mathbb{R}$
of the maximal torus $T \subset G_x$.
For example, if $G_x$ is simple,
then such a Weyl-invariant inner product is unique up to scaling.

Now, for two elements $\lambda, \mu \in \mathrm{CL} (\mathcal{X})$,
we would like to make sense of their pairing
under a quadratic norm on $\mathrm{CL} (\mathcal{X})$.
In general, since faces in $\mathrm{CL} (\mathcal{X})$
can have non-trivial automorphism groups,
if both $\lambda$ and~$\mu$ lie in such a face,
there may be multiple relative positions of~$\lambda$ and~$\mu$
that can give rise to different pairings.
This idea can be formulated as follows.

\begin{definition}[Relative positions]
    \label{def-relative-position}
    For any $\lambda, \mu \in \mathrm{CL} (\mathcal{X})$,
    define a \emph{relative position} of~$\lambda$ and~$\mu$
    as a morphism of formal lattices
    \begin{equation*}
        \alpha \colon \mathbb{Z}^2 \longrightarrow \mathrm{CL} (\mathcal{X})
    \end{equation*}
    such that
    $\alpha (1, 0) = \lambda$ and
    $\alpha (0, 1) = \mu$.

    By the discussion in \cite[\S 3.2.7]{bu2025intrinsic},
    such relative positions are in bijection with preimages~$\mu_i$
    of~$\mu$ under the induced map
    $\mathrm{CL} (\mathcal{X}_\lambda) \to \mathrm{CL} (\mathcal{X})$,
    and also with preimages~$\lambda_i$
    of~$\lambda$ under the induced map
    $\mathrm{CL} (\mathcal{X}_\mu) \to \mathrm{CL} (\mathcal{X})$.

    In particular, given a quadratic norm~$\nu$ on~$\mathrm{CL} (\mathcal{X})$,
    and given a relative position corresponding to
    the preimages~$\mu_i$ and~$\lambda_i$ as above, we write
    \begin{equation*}
        \langle \lambda, \mu_i \rangle =
        \langle \lambda_i, \mu \rangle =
        \alpha^* (\nu) ((1, 0), (0, 1)) \in \mathbb{R} \ ,
    \end{equation*}
    where we see the quadratic function
    $\alpha^* (\nu) \colon \mathbb{Z}^2 \to \mathbb{R}_{\geq 0}$
    as a symmetric bilinear form.

    If~$\mathcal{X}$ is quasi-compact,
    then so are $\mathcal{X}_\lambda$ and~$\mathcal{X}_\mu$
    by the existence of the norm and by
    \cite[Proposition~3.8.2]{halpern-leistner-instability},
    and there are finitely many preimages~$\lambda_i$ and~$\mu_i$ as above by
    \cite[Lemma~3.3.12 and Theorem~6.2.3]{bu2025intrinsic},
    so the set of relative positions is finite.
\end{definition}

\subsection{Six-functor formalisms}
\label{subsec-six-functors}

In the main part of this paper,
we will work abstractly with six-functor formalisms in the sense of
\textcite{ScholzeSixFunctor}.
We list here the extra assumptions for our six-functor formalisms,
and we give some examples afterwards.

\begin{assumption}
    \label{assumption-six-functors}
    Let~$S$ be a noetherian algebraic space,
    and let $\mathsf{St}_S$ be the $2$-category of
    algebraic stacks of finite type over~$S$ with affine stabilizers.
    Suppose we have a six-functor formalism
    \begin{equation*}
        \mathcal{X} \in \mathsf{St}_S
        \quad \longmapsto \quad
        \mathcal{D} (\mathcal{X}) \in (\text{presentable stable $\infty$-categories})
    \end{equation*}
    in the sense of \textcite{ScholzeSixFunctor},
    where the functors $f^*$, $f_*$, $f^!$, and $f_!$ are defined for all morphisms
    $f \colon \mathcal{X} \to \mathcal{Y}$ in $\mathsf{St}_S$.

    For each $\mathcal{X} \in \mathsf{St}_S$,
    we have the unit object
    $\mathbb{1}_{\mathcal{X}} \in \mathcal{D} (\mathcal{X})$
    and the dualizing object
    $\omega_{\mathcal{X}} = \pi^! (\mathbb{1}_S) \in \mathcal{D} (\mathcal{X})$,
    where $\pi \colon \mathcal{X} \to S$ is the projection.
    We have the Verdier dual functor
    \begin{equation*}
        \mathbb{D}_{\mathcal{X}} =
        \calHom_{\mathcal{X}} (-, \omega_{\mathcal{X}}) \colon
        \mathcal{D} (\mathcal{X}) \longrightarrow \mathcal{D} (\mathcal{X})^{\mathrm{op}} \ ,
    \end{equation*}
    which comes with natural isomorphisms
    $\mathbb{D}_{\mathcal{X}} \, f^* \simeq
    f^! \, \mathbb{D}_{\mathcal{Y}}$ and
    $\mathbb{D}_{\mathcal{Y}} \, f_! \simeq
    f_* \, \mathbb{D}_{\mathcal{X}}$
    for all morphisms $f \colon \mathcal{X} \to \mathcal{Y}$ in $\mathsf{St}_S$,
    as a formal consequence of the six-functor formalism.

    We further assume the following extra data and properties:

    \begin{enumerate}
        \item
            \label{item-assumption-smooth-descent}
            (\emph{Smooth descent})
            $\mathcal{D}$ satisfies smooth descent.

        \item
            \label{item-assumption-a1-invariance}
            (\emph{$\mathbb{A}^1$-invariance})
            For any $\mathcal{X} \in \mathsf{St}_S$,
            let $p \colon \mathbb{A}^1_{\mathcal{X}} \to \mathcal{X}$
            be the projection.
            Then there exists an isomorphism
            given by the adjunction unit
            \begin{equation*}
                \mathrm{id} \simeq p_* \, p^* \ .
            \end{equation*}

        \item
            (\emph{Proper pushforward})
            \label{item-assumption-proper-pushforward}
            For any representable proper morphism~$f$,
            there exists an isomorphism
            \begin{equation*}
                f_* \simeq f_! \ .
            \end{equation*}

        \item
            \label{item-assumption-smooth-pullback}
            (\emph{Smooth pullback})
            For any smooth morphism~$f$ of relative dimension~$n \in \mathbb{Z}$,
            there exists an isomorphism
            \begin{equation*}
                f^! \simeq f^* \langle n \rangle \ .
            \end{equation*}
            Here, we require a distinguished invertible object
            $\mathbb{T} \in \mathcal{D} (S)$,
            and define the \emph{Tate twist} $\langle n \rangle$
            as tensoring with the pullback of~$\mathbb{T}^{\otimes n}$.

        \item
            (\emph{Localization})
            \label{item-assumption-localization}
            For any closed immersion $i \colon \mathcal{Z} \to \mathcal{X}$
            with complement $j \colon \mathcal{U} \to \mathcal{X}$,
            there exists an exact triangle
            \begin{equation*}
                i_* \, i^! \longrightarrow
                \mathrm{id} \longrightarrow
                j_* \, j^! \ ,
            \end{equation*}
            where the morphisms are adjunction units and counits using the identifications
            $i_* \simeq i_!$ and $j^! \simeq j^*$.

        \item
            \label{item-assumption-finite-subcategories}
            (\emph{Finite subcategories})
            Optionally, we assume that for each $\mathcal{X} \in \mathsf{St}_S$,
            there is a stable full subcategory
            \begin{equation*}
                \mathcal{D}_\mathrm{c} (\mathcal{X}) \subset \mathcal{D} (\mathcal{X}) \ ,
            \end{equation*}
            closed under Tate twists,
            such that when~$\mathcal{X}$ is smooth and satisfies
            \cref{assumption-stratification} below,
            it is preserved by all induction and restriction functors
            in \cref{def-ind-res} below.
    \end{enumerate}
    For simplicity,
    we do not require compatibility
    between the isomorphisms in
    \cref{item-assumption-proper-pushforward}
    and \cref{item-assumption-smooth-pullback}
    for different~$f$,
    and between the triangles in \cref{item-assumption-localization},
    only requiring their existence.
    A more natural definition might require
    a canonical choice of such isomorphisms and triangles
    satisfying compatibility with higher homotopy coherence.
\end{assumption}

\begin{definition}[Induction and restriction]
    \label{def-ind-res}
    Under \cref{assumption-six-functors},
    let~$\mathcal{X} \in \mathsf{St}_S$,
    and let $\lambda \in \mathrm{CL} (\mathcal{X})$.
    Suppose that~$\mathcal{X}$ has quasi-compact graded points,
    so that $\mathcal{X}_\lambda$ and $\mathcal{X}_\lambda^+$ are quasi-compact.
    Consider the diagram
    \begin{equation*}
        \mathcal{X}_\lambda
        \overset{q_\lambda}{\longleftarrow}
        \mathcal{X}_\lambda^+
        \overset{p_\lambda}{\longrightarrow}
        \mathcal{X}
    \end{equation*}
    as in \cref{def-cl}.
    We define the parabolic induction and restriction functors
    \begin{alignat*}{2}
        \mathrm{Ind}_\lambda
        & =
        (p_\lambda)_! \, q_\lambda^* \colon
        & \mathcal{D} (\mathcal{X}_\lambda)
        & \longrightarrow
        \mathcal{D} (\mathcal{X}) \ ,
        \\
        \mathrm{Ind}'_\lambda
        & =
        (p_\lambda)_* \, q_\lambda^! \colon
        & \mathcal{D} (\mathcal{X}_\lambda)
        & \longrightarrow
        \mathcal{D} (\mathcal{X}) \ ,
        \\
        \mathrm{Res}_\lambda
        & =
        (q_\lambda)_! \, p_\lambda^* \colon
        & \mathcal{D} (\mathcal{X})
        & \longrightarrow
        \mathcal{D} (\mathcal{X}_\lambda) \ ,
        \\
        \mathrm{Res}'_\lambda
        & =
        (q_\lambda)_* \, p_\lambda^! \colon
        & \mathcal{D} (\mathcal{X})
        & \longrightarrow
        \mathcal{D} (\mathcal{X}_\lambda) \ .
    \end{alignat*}
    We have by definition the adjunctions
    $\mathrm{Ind}_\lambda \dashv \mathrm{Res}'_\lambda$
    and
    $\mathrm{Res}_\lambda \dashv \mathrm{Ind}'_\lambda$.

    Similarly, for any morphism
    $\alpha \overset{\sigma}{\to} \beta$
    in $\mathsf{Hall} (\mathcal{X})$,
    we have the diagram
    \begin{equation*}
        \mathcal{X}_\beta
        \overset{q_\sigma}{\longleftarrow}
        \mathcal{X}_\sigma^+
        \overset{p_\sigma}{\longrightarrow}
        \mathcal{X}_\alpha
    \end{equation*}
    as in \cref{def-hall-cat},
    and we define the parabolic induction and restriction functors
    \begin{alignat*}{2}
        \mathrm{Ind}_\sigma
        & =
        (p_\sigma)_! \, q_\sigma^* \colon
        & \mathcal{D} (\mathcal{X}_\beta)
        & \longrightarrow
        \mathcal{D} (\mathcal{X}_\alpha) \ ,
        \\
        \mathrm{Ind}'_\sigma
        & =
        (p_\sigma)_* \, q_\sigma^! \colon
        & \mathcal{D} (\mathcal{X}_\beta)
        & \longrightarrow
        \mathcal{D} (\mathcal{X}_\alpha) \ ,
        \\
        \mathrm{Res}_\sigma
        & =
        (q_\sigma)_! \, p_\sigma^* \colon
        & \mathcal{D} (\mathcal{X}_\alpha)
        & \longrightarrow
        \mathcal{D} (\mathcal{X}_\beta) \ ,
        \\
        \mathrm{Res}'_\sigma
        & =
        (q_\sigma)_* \, p_\sigma^! \colon
        & \mathcal{D} (\mathcal{X}_\alpha)
        & \longrightarrow
        \mathcal{D} (\mathcal{X}_\beta) \ .
    \end{alignat*}
    We have the adjunctions
    $\mathrm{Ind}_\sigma \dashv \mathrm{Res}'_\sigma$
    and
    $\mathrm{Res}_\sigma \dashv \mathrm{Ind}'_\sigma$.

    In the special case when $\mathcal{X}$ is smooth and \emph{$\Theta$-reductive}
    in the sense of \cite[Definition~3.10]{alper-halpern-leistner-heinloth-2023},
    which we will assume in the main part of this paper,
    the morphisms $p_\lambda$ and $p_\sigma$ are proper,
    and $q_\lambda$ and $q_\sigma$ are smooth
    for any~$\lambda$ and~$\sigma$,
    so that by the axioms \cref{item-assumption-proper-pushforward,,item-assumption-smooth-pullback}, we have
    \begin{equation}
        \label{eq-ind-ind-prime}
        \mathrm{Ind}'_\lambda \simeq
        \mathrm{Ind}_\lambda \langle \dim q_\lambda \rangle \ ,
        \qquad
        \mathrm{Ind}'_\sigma \simeq
        \mathrm{Ind}_\sigma \langle \dim q_\sigma \rangle \ .
    \end{equation}
\end{definition}

\begin{example}[Constructible sheaves]
    \label{eg-constructible-sheaves}
    Take $S = \operatorname{Spec} \mathbb{C}$,
    and let~$\Lambda$ be a field.
    For $\mathcal{X} \in \mathsf{St}_{\mathbb{C}}$, let
    \begin{equation*}
        \mathcal{D} (\mathcal{X}) =
        \mathsf{Shv} (\mathcal{X}; \Lambda) \ ,
        \qquad
        \mathcal{D}_\mathrm{c} (\mathcal{X}) =
        \mathsf{Shv}_{\mathrm{con}} (\mathcal{X}; \Lambda)
    \end{equation*}
    be the categories of ind-constructible and
    (bounded) constructible sheaves
    on~$\mathcal{X}$, respectively;
    the former is defined by taking the ind-completion of latter on schemes,
    then passing to stacks via smooth descent.
    The Tate object is $\mathbb{T} = \Lambda [2]$.
    See \textcite[\S A.3]{sandvik};
    see also \textcite[Appendices E--G]{agkrrv}.

    Note that although
    $\mathsf{Shv}_\mathrm{cpt} (\mathcal{X})
    \subset \mathsf{Shv}_{\mathrm{con}} (\mathcal{X})$,
    where $(-)_\mathrm{cpt}$ denotes the subcategory of compact objects,
    the two categories are not the same
    when~$\mathcal{X}$ admits non-unipotent stabilizers
    (cf.~\cite[\S 10.2.1]{drinfeld-gaitsgory-2013-finiteness}).
    For example,
    the constant sheaf on~$\mathrm{B} \mathbb{G}_\mathrm{m}$
    is constructible but not compact.
\end{example}

\begin{example}[Mixed Hodge modules]
    \label{eg-mhm}
    Take again $S = \operatorname{Spec} \mathbb{C}$,
    and for each $\mathcal{X} \in \mathsf{St}_{\mathbb{C}}$, let
    \begin{equation*}
        \mathcal{D} (\mathcal{X}) =
        \mathsf{MHM} (\mathcal{X}; \mathbb{Q}) \ ,
        \qquad
        \mathcal{D}_\mathrm{c} (\mathcal{X}) =
        \mathsf{MHM}_{\mathrm{con}} (\mathcal{X}; \mathbb{Q})
    \end{equation*}
    be the categories of all (ind-constructible) mixed Hodge modules and
    (bounded) constructible mixed Hodge modules on~$\mathcal{X}$,
    respectively, defined as in
    \textcite[Theorem~4.1.1]{tubach-2025-mixed}.
    The Tate object is
    $\mathbb{T} = \mathbb{Q} (1) [2]$.
    Similarly to
    \cref{eg-constructible-sheaves},
    the compact objects of $\mathsf{MHM} (\mathcal{X})$
    are contained in $\mathsf{MHM}_{\mathrm{con}} (\mathcal{X})$,
    but the two categories are different in general.
\end{example}

\begin{example}[$\ell$-adic sheaves]
    \label{eg-ell-adic-sheaves}
    Let $S$ be a quasi-excellent finite-dimensional scheme,
    and let~$\ell$ be a prime number which is invertible on~$S$.
    Let~$\Lambda$ be one of
    $\mathbb{Z}_\ell$, $\mathbb{Q}_\ell$, or~$\overline{\mathbb{Q}}_\ell$.
    For $\mathcal{X} \in \mathsf{St}_S$, let
    \begin{equation*}
        \mathcal{D} (\mathcal{X}) =
        \mathsf{Shv} (\mathcal{X}; \Lambda) \ ,
        \qquad
        \mathcal{D}_\mathrm{c} (\mathcal{X}) =
        \mathsf{Shv}_{\mathrm{con}} (\mathcal{X}; \Lambda)
    \end{equation*}
    be the categories of ind-constructible and (bounded) constructible
    étale sheaves on~$\mathcal{X}$ with coefficients in~$\Lambda$,
    respectively.
    There is a six-functor formalism of all étale sheaves constructed by
    \textcite{liu-zheng},
    and the ind-constructible and constructible subcategories are defined as in
    \textcite{hemo-richarz-scholbach-2023} for schemes,
    and via smooth descent for stacks.
    The proper pushforward and smooth pullback properties follow from
    \cite[Theorems~7.2.6 and~7.2.7]{liu-zheng}, with the twist
    $\mathbb{T} = \Lambda (1) [2]$.
    See also \textcite[Appendices~E--F]{agkrrv}
    for the case when $S = \operatorname{Spec} k$ for a field~$k$.
\end{example}

\begin{example}[Holonomic D-modules]
    \label{eg-holonomic-d-modules}
    Let $k$ be a field of characteristic~$0$.
    For $\mathcal{X} \in \mathsf{St}_k$, let
    \begin{equation*}
        \mathcal{D} (\mathcal{X}) =
        \mathsf{DMod}_{\textnormal{ind-hol}} (\mathcal{X}) \ ,
        \qquad
        \mathcal{D}_\mathrm{c} (\mathcal{X}) =
        \mathsf{DMod}_{\mathrm{hol}} (\mathcal{X})
    \end{equation*}
    be the categories of ind-holonomic and
    (bounded) holonomic D-modules on~$\mathcal{X}$,
    respectively,
    where the former is defined by taking the ind-completion of latter on schemes,
    then passing to stacks via smooth descent.
    See \textcite[Remark~8.38]{ScholzeSixFunctor}
    for the existence of this six-functor formalism on schemes,
    which then extends to stacks by \cite[Proposition~5.16]{ScholzeSixFunctor}.
    See, for example, \textcite[Proposition~2.2]{koppensteiner-2020-hochschild}
    for the proper pushforward and smooth pullback properties,
    where $\mathbb{T} = k [2]$.
    See also \textcite[Appendices~E--F]{agkrrv}.
\end{example}

\begin{example}[D-modules]
    \label{eg-d-modules}
    The following example does not fully satisfy
    \cref{assumption-six-functors},
    in that we do not always define $f_!$ and $f^*$,
    but those conditions that do not involve these functors are still satisfied.

    Let $k$ be a field of characteristic~$0$,
    and for $\mathcal{X} \in \mathsf{St}_k$, let
    \begin{equation*}
        \mathcal{D} (\mathcal{X}) =
        \mathsf{DMod} (\mathcal{X})
        \ , \qquad
        \mathcal{D}_\mathrm{c} (\mathcal{X}) =
        \mathsf{DMod}_{\mathrm{coh}} (\mathcal{X})
    \end{equation*}
    be the categories of all D-modules and
    (bounded) coherent D-modules on~$\mathcal{X}$, respectively.
    Then for a morphism $f$ in $\mathsf{St}_k$,
    $f_*$ and $f^!$ are always defined on $\mathcal{D} (\mathcal{X})$;
    we define~$f_!$ (resp.~$f^*$) as the left adjoint of
    $f^!$ (resp.~$f_*$) if it exists.
    If~$f$ is proper (resp.~smooth),
    then $f_!$ (resp.~$f^*$) is always defined and satisfies
    \cref{assumption-six-functors}~\cref{item-assumption-proper-pushforward}
    (resp.~\cref{item-assumption-smooth-pullback}),
    where the twist is $\mathbb{T} = k [2]$.

    All induction and restriction functors in
    \cref{def-ind-res}
    possibly except $\mathrm{Ind}$ are defined,
    where $\mathrm{Ind}'$ and $\mathrm{Res}'$
    only use the $!$-pull and the $*$-push,
    and $\mathrm{Res}$ exists by Braden's theorem,
    as we will discuss in \cref{def-braden-theorem} and below.
    If, moreover, $\mathcal{X}$ is smooth and $\Theta$-reductive,
    then $\mathrm{Ind}$ is also defined and satisfies
    \cref{eq-ind-ind-prime}.

    When $\mathcal{X}$ is smooth and $\Theta$-reductive,
    coherent D-modules are preserved by induction
    since it is a smooth pullback followed by a proper pushforward,
    and they are preserved by restrictions
    as we will show in \cref{lem-restriction-coherent}.
    See also \textcite[Corollary~3.9]{gunningham2018generalized}
    for the case when $\mathcal{X} = \mathfrak{g} / G$
    for a reductive group~$G$ and its adjoint representation~$\mathfrak{g}$.
\end{example}

\begin{definition}[Braden's theorem]
    \label{def-braden-theorem}
    In the situation of \cref{assumption-six-functors},
    let~$\mathcal{X} \in \mathsf{St}_S$ be a stack.
    We say that the six-functor formalism~$\mathcal{D}$
    satisfies \emph{Braden's theorem} for~$\mathcal{X}$,
    if for any morphism
    $\alpha \overset{\sigma}{\to} \beta$ in $\mathsf{Hall} (\mathcal{X})$,
    in the diagram
    \begin{equation*}
        \begin{tikzcd}[row sep={3em, between origins}, column sep={5em, between origins}]
            & {\mathcal{X}_{\sigma}^+}
            \ar[dl, shift right, "q_\sigma"']
            \ar[dr, "p_\sigma"]
            &
            \\
            \mathcal{X}_\beta
            \ar[ur, shift right, "\imath_\sigma"']
            \ar[dr, shift left, "\imath_{-\sigma}"]
            &&
            \mathcal{X}_\alpha
            \\
            & \mathcal{X}_{\sigma}^- \rlap{ ,}
            \ar[ul, shift left, "q_{-\sigma}"]
            \ar[ur, "p_{-\sigma}"']
        \end{tikzcd}
    \end{equation*}
    where $\mathcal{X}_\sigma^- = \mathcal{X}_{-\sigma}^+$,
    we have natural isomorphisms
    \begin{equation*}
        (q_{\sigma})_* \, p_{\sigma}^!
        \simeq (q_{-\sigma})_! \, p_{-\sigma}^* \ ,
        \qquad
        (q_\sigma)_* \simeq \imath_{\sigma}^* \ ,
        \qquad
        (q_\sigma)_! \simeq \imath_{\sigma}^! \ ,
    \end{equation*}
    where the first morphism is constructed from the six-functor formalism,%
    \footnote{
        See \textcite[Proof of Proposition~A.1]{kinjo2024cohomological}
        for the construction of this map in the setting of constructible sheaves,
        which can be adapted to other six-functor formalisms.
        Although they use the generalized base change transform constructed in
        \cite[\S 7.1.3]{kinjo2024cohomological},
        which requires
        \cref{assumption-six-functors}~\cref{item-assumption-smooth-pullback}
        for quasi-smooth morphisms,
        it is only applied to an étale morphism in our situation, by
        \cite[Lemma~A.2]{kinjo2024cohomological}.
    }
    and the other two are induced by the adjunction unit
    $\mathrm{id} \to (\imath_\sigma)_* \, \imath_\sigma^*$
    and the counit
    $(\imath_\sigma)_! \, \imath_\sigma^! \to \mathrm{id}$,
    respectively.
\end{definition}

\noindent
See \textcite{drinfeld-gaitsgory-2014-braden} for background on this,
and \textcite{braden-2003-hyperbolic} for the original reference.

For example, Braden's theorem holds for all stacks
$\mathcal{X} \in \mathsf{St}_k$ with $\operatorname{char} k = 0$
and all sheaf theories in
\cref{eg-constructible-sheaves,eg-mhm,eg-ell-adic-sheaves,eg-holonomic-d-modules}
by \textcite[Propositions~A.1 and~7.2, Remark~A.3]{kinjo2024cohomological}.

For all D-modules, as in \cref{eg-d-modules},
the $*$-pullback and the $!$-pushforward are in general only partially defined:
their images naturally land in the pro-category of D-modules,
as in \cite[\S 3]{drinfeld-gaitsgory-2014-braden}.
However, \textcite[Theorem~C.5.3]{drinfeld-gaitsgory-2015-compact}
together with the argument of \cite[Proposition~A.1]{kinjo2024cohomological}
implies that the ones appearing in \cref{def-braden-theorem}
are always fully defined,
and that all isomorphisms in \cref{def-braden-theorem}
hold for all stacks in $\mathsf{St}_k$.
In particular, the functors $\mathrm{Ind}'$, $\mathrm{Res}$, and $\mathrm{Res}'$
in \cref{def-ind-res}
are defined,
and for a smooth and $\Theta$-reductive stack,
$\mathrm{Ind}$ is also defined and satisfies
\cref{eq-ind-ind-prime}.

\begin{lemma}[Restriction preserves coherent D-modules]
    \label{lem-restriction-coherent}
    Let~$k$ be a field of characteristic~$0$,
    and let~$\mathcal{X}$ be a smooth, $\Theta$-reductive
    algebraic stack of finite type over~$k$
    with affine stabilizers and quasi-compact graded points.
    Then for any $\lambda \in \mathrm{CL} (\mathcal{X})$,
    the restriction functors
    \begin{equation*}
        \mathrm{Res}_\lambda \ ,
        \quad
        \mathrm{Res}'_\lambda \colon
        \quad
        \mathsf{DMod} (\mathcal{X}) \longrightarrow
        \mathsf{DMod} (\mathcal{X}_\lambda)
    \end{equation*}
    preserve coherent D-modules.
\end{lemma}

\begin{proof}
    First, note that the restriction functors preserve compact D-modules,
    defined as compact objects in $\mathsf{DMod} (-)$,
    and have finite cohomological amplitude,
    since by Braden's theorem and
    \cref{prop-adjunction-seq} below,
    they have adjoints on both sides given by induction functors,
    which are smooth pullbacks followed by proper pushforwards,
    and hence are continuous and have finite cohomological amplitude.
    Now, by \textcite[Lemma~3.5.3]{drinfeld-gaitsgory-2015-compact},
    for any coherent D-module~$E \in \mathsf{DMod}_{\mathrm{coh}} (\mathcal{X})$
    and any integer~$n$,
    there exists a compact D-module~$E' \in \mathsf{DMod}_{\mathrm{cpt}} (\mathcal{X})$
    and a morphism $E' \to E$ whose cone lies in
    $\mathsf{DMod} (\mathcal{X})^{< -n}$.
    For $n$ large enough, it follows that
    $\mathrm{Res}_\lambda (E) \simeq \tau^{\geq -m} \, \mathrm{Res}_\lambda (E')$
    for some~$m$, where $\mathrm{Res}_\lambda (E')$ is compact, hence coherent,
    and $\tau^{\geq -m}$ preserves coherence as in
    \cite[below Lemma~3.5.3]{drinfeld-gaitsgory-2015-compact},
    so that $\mathrm{Res}_\lambda (E)$ is coherent.
    The same argument applies to $\mathrm{Res}'_\lambda$.
\end{proof}

\noindent
Finally, we collect a standard result on $\mathbb{A}^1$-homotopy invariance
that will be used later.

\begin{lemma}
    \label{lem-a1-contract}
    Under \cref{assumption-six-functors},
    let $f \colon \mathcal{X} \to \mathcal{Y}$
    be a morphism in $\mathsf{St}_S$
    which is an $\mathbb{A}^1$-deformation retract.
    Then there is an isomorphism given by the adjunction unit
    \begin{equation*}
        \mathrm{id} \simeq f_* \, f^* \ .
    \end{equation*}
    In particular, if~$f$ is also smooth of relative dimension~$n$, then
    $f_* \, f^! \simeq \langle n \rangle$.
\end{lemma}

\begin{proof}
    It is enough to show that the left adjoint~$f^*$ is fully faithful.
    Let $p \colon \mathbb{A}^1 \times \mathcal{X} \to \mathcal{X}$
    be the projection,
    and let $h \colon \mathbb{A}^1 \times \mathcal{X} \to \mathcal{X}$
    be the deformation retract,
    so that there exists $s \colon \mathcal{Y} \to \mathcal{X}$ such that
    $f s \simeq \mathrm{id}_{\mathcal{Y}}$,
    $h i_0 \simeq s f$,
    $h i_1 \simeq \mathrm{id}_{\mathcal{X}}$,
    and $f h \simeq f p$,
    where $i_a \colon \mathcal{X} \to \mathbb{A}^1 \times \mathcal{X}$
    is the inclusion as $\{ a \} \times \mathcal{X}$.
    Now $s^* \, f^* \simeq \mathrm{id}$,
    and for $x, y \in \mathcal{D} (\mathcal{Y})$, consider the morphism
    \begin{align}
        \mathrm{Hom} (f^* \, x, f^* \, y)
        \overset{h^*}{\to} {}
        & \mathrm{Hom} (h^* \, f^* \, x, h^* \, f^* \, y)
        \notag \\ \simeq {}
        & \mathrm{Hom} (p^* \, f^* \, x, p^* \, f^* \, y)
        \notag \\ \simeq {}
        & \mathrm{Hom} (f^* \, x, f^* \, y) \ ,
        \label{eq-proof-a1-retract-invariance}
    \end{align}
    where the last step uses the full faithfulness of~$p^*$ by
    \cref{assumption-six-functors}~\cref{item-assumption-a1-invariance}.
    It follows that the last morphism agrees with both
    the ones induced by $i_0^*$ and $i_1^*$,
    so that the composition map in
    \cref{eq-proof-a1-retract-invariance} is
    $\mathrm{id} \simeq i_1^* \, h^* \simeq i_0^* \, h^* \simeq f^* \, s^*$
    on the hom-space.
    In particular, $s^*$ and~$f^*$ induce inverse isomorphisms
    $\mathrm{Hom} (x, y) \simeq \mathrm{Hom} (f^* \, x, f^* \, y)$,
    so that $f^*$ is fully faithful.
\end{proof}

%% file: steinberg.tex
\subsection{Definition and examples}

Throughout, let~$S$ be a locally noetherian algebraic space,
and let $\mathcal{X}$ be an algebraic stack locally of finite type over~$S$
with affine stabilizers and separated diagonal.

\begin{definition}[Steinberg stack]
    The \emph{Steinberg stack} of~$\mathcal{X}$ is defined as the fibre product
    \begin{equation*}
        \calSt (\mathcal{X}) =
        \calFilt (\mathcal{X})
        \underset{\mathcal{X}}{\times}
        \calFilt (\mathcal{X}) \ .
    \end{equation*}
    For $\lambda, \mu \in \mathrm{CL} (\mathcal{X})$,
    define the open and closed substack
    \begin{equation*}
        \calSt_{\lambda, \mu} =
        \calSt (\mathcal{X})_{\lambda, \mu} =
        \mathcal{X}_\lambda^+
        \underset{\mathcal{X}}{\times}
        \mathcal{X}_\mu^+
        \subset \calSt (\mathcal{X}) \ .
    \end{equation*}
    We have a commutative diagram
    \begin{equation*}
        \begin{tikzcd}[column sep={8em, between origins}]
            \calFilt^2 (\mathcal{X})_{\lambda, \mu}
            \ar[r] \ar[d, hook]
            \ar[dr, start anchor=-45, end anchor=135, pos=.2, phantom, "\ulcorner"]
            &
            \calSt_{\lambda, \mu}
            \ar[r] \ar[d, hook]
            \ar[dr, start anchor=-45, end anchor=135, pos=.2, phantom, "\ulcorner"]
            &
            \mathcal{X}_\lambda^+ \underset{S}{\times} \mathcal{X}_\mu^+
            \ar[d, hook]
            \\
            \calFilt^2 (\mathcal{X})
            \ar[r]
            &
            \calSt (\mathcal{X})
            \ar[r]
            &
            \calFilt (\mathcal{X}) \underset{S}{\times} \calFilt (\mathcal{X})
            \rlap{ ,}
        \end{tikzcd}
    \end{equation*}
    where the right vertical arrow is the inclusion of an open and closed substack,
    and the bottom left horizontal arrow is induced by the two maps
    $\mathbb{A}^1 / \mathbb{G}_\mathrm{m} \to (\mathbb{A}^1 / \mathbb{G}_\mathrm{m})^2$
    given by $(\mathrm{id}, 1)$ and $(1, \mathrm{id})$, respectively.
    The stack $\calFilt^2 (\mathcal{X})_{\lambda, \mu}$
    is the disjoint union of the connected components
    $(\mathcal{X}_\lambda^+)_{\mu_i}^+ \subset \calFilt^2 (\mathcal{X})$
    for relative positions $(\lambda, \mu_i)$ of~$\lambda$ and~$\mu$
    as in \cref{def-relative-position},
    where we identify
    $\mathrm{CL} (\mathcal{X}_\lambda) \simeq \mathrm{CL} (\mathcal{X}_\lambda^+)$;
    see \cite[\S\S 3.2.7 and~5.1.9]{bu2025intrinsic}.

    More generally, for a pair of morphisms
    $\beta \overset{\sigma}{\gets} \alpha \overset{\tau}{\to} \gamma$
    in the Hall category $\mathsf{Hall} (\mathcal{X})$, we define the Steinberg stack
    \begin{equation*}
        \calSt_{\sigma, \tau}
        =
        \calSt (\mathcal{X}_\alpha)_{\sigma, \tau}
        =
        \mathcal{X}_\sigma^+
        \underset{\mathcal{X}_\alpha}{\times}
        \mathcal{X}_\tau^+ \ .
    \end{equation*}
    Under certain finiteness assumptions,
    such as when~$\mathcal{X}$ is quasi-compact and has quasi-compact graded points
    as in \cref{def-qc-graded-points},
    by \cite[Theorem~6.2.3]{bu2025intrinsic},
    for generic elements
    $\lambda \in \sigma \subset \mathrm{CL} (\mathcal{X}_\alpha)$
    and
    $\mu \in \tau \subset \mathrm{CL} (\mathcal{X}_\alpha)$,
    we have a canonical isomorphism
    \begin{equation*}
        \calSt (\mathcal{X}_\alpha)_{\sigma, \tau}
        \simeq \calSt (\mathcal{X}_\alpha)_{\lambda, \mu} \ .
    \end{equation*}
\end{definition}

\begin{remark}
    \label{rem-mapping-stack}
    Since we have the pushout
    \begin{equation*}
        \mathbb{A}^1 / \mathbb{G}_{\mathrm{m}}
        \underset{\{ 1 \}}{\cup}
        \mathbb{A}^1 / \mathbb{G}_{\mathrm{m}}
        \simeq
        (\mathbb{A}^2 \setminus \{ 0 \}) / \mathbb{G}_{\mathrm{m}}^2
        \simeq
        \mathbb{P}^1 / \mathbb{G}_{\mathrm{m}} \ ,
    \end{equation*}
    with $\mathbb{G}_{\mathrm{m}}^2$ acting on~$\mathbb{A}^2$
    by scaling each coordinate
    and~$\mathbb{G}_{\mathrm{m}}$ acting on~$\mathbb{P}^1$ by scaling,
    we have
    \begin{equation*}
        \calSt (\mathcal{X}) \simeq
        \calMap (\mathbb{P}^1 / \mathbb{G}_{\mathrm{m}}, \mathcal{X}) \ .
    \end{equation*}
\end{remark}

\begin{example}[Classifying stacks]
    \label{eg-steinberg-classifying}
    Let~$k$ be a field, and let~$G$ be an affine algebraic group over~$k$
    with a split maximal torus $T \subset G$.
    Then, for cocharacters $\lambda, \mu \colon \mathbb{G}_\mathrm{m} \to T$, we have
    \begin{equation*}
        \calSt (\mathrm{B} G)_{\lambda, \mu}
        \simeq
        \mathrm{B} P_\lambda
        \underset{\mathrm{B} G}{\times}
        \mathrm{B} P_\mu
        \simeq
        P_\lambda \backslash G / P_\mu \ ,
    \end{equation*}
    where $P_\lambda, P_\mu \subset G$
    are parabolic subgroups defined as attracting loci in~$G$
    for the conjugation actions by~$\lambda, \mu$, respectively.

    In particular, when~$G$ is reductive,
    $\bar{k}$-points of $P_\lambda \backslash G / P_\mu$
    are in bijection with Bruhat cells of~$G$,
    and also with elements of the double quotient
    $W_\lambda \backslash W / W_\mu$,
    where $W$ is the Weyl group
    and $W_\lambda, W_\mu \subset W$
    are the stabilizers of~$\lambda$, $\mu$, respectively.
\end{example}

\begin{example}[Adjoint quotients]
    \label{eg-steinberg-adjoint}
    Let~$G$ and $\lambda, \mu$ be as above. The Steinberg stacks
    \begin{equation*}
        \calSt (\mathfrak{g} / G)_{\lambda, \mu} \simeq
        \mathfrak{p}_\lambda / P_\lambda
        \underset{\mathfrak{g} / G}{\times}
        \mathfrak{p}_\mu / P_\mu
        \qquad\text{and}\qquad
        \calSt (G / G)_{\lambda, \mu} \simeq
        P_\lambda / P_\lambda
        \underset{G / G}{\times}
        P_\mu / P_\mu
    \end{equation*}
    are of interest in geometric representation theory,
    where~$G$ acts on~$\mathfrak{g}$ and~$G$ by conjugation,
    and $\mathfrak{g}, \mathfrak{p}_\lambda, \mathfrak{p}_\mu$
    are the Lie algebras of~$G$, $P_\lambda$, $P_\mu$.
    For example, when~$G$ is reductive,
    $\calSt (G / G)_{\lambda, \mu}$
    gives a geometric description and categorification of affine Hecke algebras,
    as in \textcite{bezrukavnikov-2016-aha},
    and is involved in the tame local categorical Langlands correspondence
    as in \textcite[\S 2]{zhu-tame}.

    When~$G$ is split reductive,
    the stacks $\calSt (\mathfrak{g} / G)_{\lambda, \mu}$
    and $\calSt (G / G)_{\lambda, \mu}$
    have irreducible components indexed by the double quotient
    $W_\lambda \backslash W / W_\mu$,
    by \cref{cor-irred-comp}.

    For example, when $G = \mathrm{SL}_2$,
    if $\lambda \colon \mathbb{G}_\mathrm{m} \to G$
    is a generic cocharacter, then
    $\calSt (\mathfrak{sl}_2 / \mathrm{SL}_2)_{\lambda, \lambda} \simeq Z / \mathrm{SL}_2$,
    where
    \begin{equation*}
        Z = \{(x, \ell_1, \ell_2) \in \mathfrak{sl}_2 \times \mathbb{P}^1 \times \mathbb{P}^1
        \mid x(\ell_i) \subset \ell_i\} \ .
    \end{equation*}
    If we write $Z_1 = \{\ell_1 \neq \ell_2\} \subset Z$
    and $Z_2 = \{\ell_1 = \ell_2\} \subset Z$,
    so that $Z = Z_1 \sqcup Z_2$
    with $Z_1$ open and $Z_2$ closed,
    then $Z_1$ is an affine bundle of rank $1$ over
    $(\mathbb{P}^1 \times \mathbb{P}^1) \setminus \Delta$,
    and $Z_2$ is an affine bundle of rank $2$ over
    $\mathbb{P}^1$.
    The two irreducible components of $Z$
    are $\overline{Z_1}$ and $Z_2$,
    which both have dimension~$3$
    and intersect along a $2$-dimensional subvariety of~$Z$.
\end{example}

\begin{example}[$G$-bundles]
    \label{eg-steinberg-bun}
    Let $G$ and $\lambda, \mu$ as above,
    let~$X$ be a proper scheme over~$k$,
    and let $\calBun_G (X) = \calMap (X, \mathrm{B} G)$
    be the moduli stack of $G$-bundles on~$X$.
    We have the Steinberg stack
    \begin{equation*}
        \calSt (\calBun_G (X))_{\lambda, \mu}
        \simeq
        \calBun_{P_\lambda} (X)
        \underset{\calBun_G (X)}{\times}
        \calBun_{P_\mu} (X)
        \simeq
        \calMap (X, P_\lambda \backslash G / P_\mu) \ .
    \end{equation*}
    Here, we have abused notation by writing
    $\calSt (\calBun_G (X))_{\lambda, \mu}$
    for cocharacters $\lambda, \mu$,
    so it is still an open and closed substack of
    $\calSt (\calBun_G (X))$,
    but it may have multiple connected components.

    When $X = C$ is a geometrically connected
    smooth projective curve and~$G$ is split reductive,
    we also consider the semistable version
    \begin{equation*}
        \calSt (\calBun_G^\mathrm{ss} (C))_{\lambda, \mu}
        \simeq
        \calBun_{P_\lambda}^\mathrm{ss} (C)
        \underset{\calBun_G^\mathrm{ss} (C)}{\times}
        \calBun_{P_\mu}^\mathrm{ss} (C) \ ,
    \end{equation*}
    where we regard a $P_\lambda$-bundle as semistable
    if its associated graded $L_\lambda$-bundle is semistable
    and has a \emph{$G$-admissible} degree in the sense of
    \cite[\S 10.3.1]{bu-davison-ibanez-nunez-kinjo-padurariu},
    where admissibility ensures semistability of the induced $G$-bundle.
\end{example}

\subsection{The Bruhat stratification}

In this subsection, we prove \cref{theorem:stratification},
which is one of the main results of this paper.
It gives a decomposition of the Steinberg stack~$\calSt (\mathcal{X})$
generalizing the Bruhat decomposition of flag varieties of reductive groups.

\begin{definition}[Stratification]
    \label{def-stratification}
    A \emph{stratification} of a stack~$\mathcal{X}$
    is a partition
    \begin{equation*}
        \mathcal{X} = \bigcup_{i \in I} \mathcal{X}_i \ ,
    \end{equation*}
    where $I$ is a totally ordered set,
    $\mathcal{X}_i \subset \mathcal{X}$
    are disjoint locally closed substacks, such that
    $\mathcal{X}_{\geq i} = \bigcup_{j \geq i} \mathcal{X}_j$
    is open in~$\mathcal{X}$ for all $i \in I$.
    In this case, we also say that the morphism
    $\coprod_{i \in I} \mathcal{X}_i \to \mathcal{X}$
    is a stratification.

    If, moreover, we have
    $\mathcal{X}_i \subset \overline{\mathcal{X}_j}$
    whenever $i \leq j$,
    then we say that the stratification is a \emph{good stratification}.
\end{definition}

\begin{assumption}
    \label{assumption-stratification}
    Let~$S$ be a noetherian algebraic space,
    and let~$\mathcal{X}$ be an algebraic stack of finite type over~$S$
    with affine stabilizers and separated diagonal.
    Assume we are in either one of the following situations:

    \begin{enumerate}
        \item
            \label{item-case-gms}
            $S$ is of characteristic~$0$,
            $\mathcal{X}$ has a good moduli space~$X$,
            and $\mathrm{CL} (\mathcal{X})$ has a quadratic norm
            in the sense of \cref{def-norm};

        \item
            \label{item-case-glob-quot}
            $S$ is arbitrary,
            $\mathcal{X}$ has an adequate moduli space
            $X$ of finite type over~$S$,
            and it has a global quotient presentation
            $\mathcal{X} = Y / \mathrm{GL}_n$
            for an algebraic space~$Y$.
    \end{enumerate}
\end{assumption}

\noindent
Note that having a global quotient presentation implies
the existence of a quadratic norm on~$\mathrm{CL} (\mathcal{X})$,
since we may pull back the norm on $\mathrm{BGL}_n$
along the representable morphism~$\mathcal{X} \to \mathrm{BGL}_n$,
by \cite[Lemma~3.3.12]{bu2025intrinsic}.
In particular, in both cases,
$\mathcal{X}$ has quasi-compact graded points as in
\cref{def-qc-graded-points}.
Moreover, in both cases,
$\mathcal{X}$ is $\Theta$-reductive by
\cite[Proposition~3.21]{alper-halpern-leistner-heinloth-2023}.

\begin{theorem}[Bruhat stratification]
    \label{theorem:stratification}
    Under \cref{assumption-stratification},
    the natural morphism
    \begin{equation}
        \label{eq-steinberg-strat}
        \calFilt^2 (\mathcal{X}) \longrightarrow
        \calSt (\mathcal{X})
    \end{equation}
    is a stratification.

    In particular, for any
    $\lambda, \mu \in \mathrm{CL} (\mathcal{X})$,
    the morphism
    \begin{equation}
        \calFilt^2 (\mathcal{X})_{\lambda, \mu} =
        \coprod_i {} (\mathcal{X}_{\lambda}^+)_{\mu_i}^+
        \longrightarrow
        \calSt_{\lambda, \mu}
    \end{equation}
    is a stratification,
    where~$i$ runs through relative positions of~$\lambda$ and~$\mu$
    as in \cref{def-relative-position}.

    In the case \cref{item-case-gms} in
    \cref{assumption-stratification},
    the order of the stratification can be chosen to be
    any total order refining the order of the inner product
    $-\langle \lambda, \mu_i \rangle \in \mathbb{R}$.
    In the case \cref{item-case-glob-quot},
    the order is induced by the Bruhat order for~$\mathrm{GL}_n$,
    as we will explain in the proof.
\end{theorem}

\begin{remark}
    In light of \cref{rem-mapping-stack},
    \cref{theorem:stratification}
    can be seen as a refinement of the fact that
    if~$\mathcal{X}$ is $\Theta$-reductive
    (which is implied by \cref{assumption-stratification}),
    then for any field~$k$, any morphism
    $(\mathbb{A}^2_{\smash{k}} \setminus \{ 0 \}) / \mathbb{G}_{\smash{\mathrm{m}, \, k}}^2 \to \mathcal{X}$
    extends uniquely to a morphism
    $\mathbb{A}^2_{\smash{k}} / \mathbb{G}_{\smash{\mathrm{m}, \, k}}^2 \to \mathcal{X}$,
    as in the proof of
    \textcite[Lemma~5.1.11]{halpern-leistner-instability}.
    In other words, the morphism
    \cref{eq-steinberg-strat}
    is universally bijective.
    We strengthen this fact by showing that
    there is a global order on the set of strata
    making it into a stratification.
\end{remark}

\begin{remark}
    The stratification in \cref{theorem:stratification}
    is not always a good stratification in the sense of \cref{def-stratification}.
    For example, for $\mathcal{X} = \mathfrak{sl}_2 / \mathrm{SL}_2$
    as in \cref{eg-steinberg-adjoint},
    the stratification corresponds to the decomposition
    $Z = Z_1 \sqcup Z_2$ there,
    which is not a good stratification since
    $Z_2 \not\subset \overline{Z_1}$.
\end{remark}

\begin{remark}
    \cref{theorem:stratification}
    does not hold for stacks without good or adequate moduli spaces in general.
    For example, consider the stack
    $\mathcal{X} = \mathrm{B} G$ over a field,
    where $G = \mathbb{G}_{\mathrm{a}}\rtimes\mathbb{G}_{\mathrm{m}}$
    with $\mathbb{G}_{\mathrm{m}}$ acting on $\mathbb{G}_{\mathrm{a}}$ with weight $1$.
    Let $\lambda \colon \mathbb{G}_{\mathrm{m}} \to G$
    be given by $t \mapsto (0, t^{-1})$,
    so $\mathcal{X}_{\lambda}^+ \simeq \mathrm{B} {\mathbb{G}_{\mathrm{m}}}$,
    and consider the Steinberg stack
    \begin{equation*}
        \calSt (\mathcal{X})_{\lambda,\lambda}
        = \mathcal{X}_{\lambda}^+ \underset{\mathcal{X}}{\times} \mathcal{X}_{\lambda}^+
        \simeq \mathbb{G}_{\mathrm{m}}\backslash G/\mathbb{G}_{\mathrm{m}}
        \simeq \mathbb{A}^1 / \mathbb{G}_\mathrm{m} \ ,
    \end{equation*}
    where the last quotient has weight~$1$.
    Since every connected component of
    $\calFilt^n (\mathcal{X})$
    for any $n \geq 0$ is isomorphic to either
    $\mathrm{B} G$ or
    $\mathrm{B} \mathbb{G}_{\mathrm{m}}$,
    the stack $\calSt (\mathcal{X})_{\lambda,\lambda}$
    cannot be stratified by any substack of
    $\calFilt^n (\mathcal{X})$ for any~$n$.
\end{remark}

\begin{remark}[$\Theta$-stratification]
    The stratification in \cref{theorem:stratification}
    in fact has the structure of a $\Theta$-stratification in the sense of
    \textcite[Definition~2.1.2]{halpern-leistner-instability}.
    To see this,
    it is enough to show that each $(\mathcal{X}_\lambda^+)_{\mu_i}^+$
    can be identified with a connected component of
    $\calFilt (\calSt_{\smash{\lambda, \mu}}^{\geq i})$
    for the open substack
    $\calSt_{\smash{\lambda, \mu}}^{\geq i} \subset \calSt_{\lambda, \mu}$.
    Indeed, we may take this component to be
    \begin{equation*}
        (\calSt_{\smash{\lambda, \mu}}^{\geq i})_{\mu_i}^+
        \subset (\calSt_{\lambda, \mu})_{\mu_i}^+
        \simeq
        (\mathcal{X}_\lambda^+)_{\mu_i}^+
        \underset{\mathcal{X}_\mu^+}{\times}
        (\mathcal{X}_\mu^+)_\mu^+
        \simeq
        (\mathcal{X}_\lambda^+)_{\mu_i}^+ \ ,
    \end{equation*}
    where
    $(\calSt_{\smash{\lambda, \mu}}^{\geq i})_{\mu_i}^+$
    and
    $(\calSt_{\lambda, \mu})_{\mu_i}^+$
    are the preimages of
    $(\mathcal{X}_\lambda^+)_{\mu_i}^+$
    under the induced maps on filtered points,
    and we used the fact that $\calFilt (-)$ commutes with limits and that
    $(\mathcal{X}_\mu^+)_\mu^+ \simeq \mathcal{X}_\mu^+$.
    In fact, this shows the stronger property that
    $(\mathcal{X}_\lambda^+)_{\mu_i}^+$
    is a connected component of
    $\calFilt (\calSt_{\lambda, \mu})$
    which is contained in
    $\calFilt (\calSt_{\smash{\lambda, \mu}}^{\geq i})$.
    Note that the $\Theta$-stratification structure is not canonical,
    since the choice of the component of
    $\calFilt (\calSt_{\smash{\lambda, \mu}}^{\geq i})$
    is also a part of the data of a $\Theta$-stratification,
    and for example, we may use~$\lambda_i$ instead of~$\mu_i$
    to pick a different component.
\end{remark}

\begin{proof}[Proof of \cref{theorem:stratification}]
    It is enough to prove the version for fixed $\lambda, \mu$.
    In fact, we may assume that $\mathcal{X} \simeq U / \mathrm{GL}_n$
    for an affine scheme $U$,
    since the statement can be checked étale locally
    over the good or adequate moduli space,
    and in the case \cref{item-case-gms},
    this follows from \textcite[Theorem~6.1]{alper-hall-rydh-2026-etale};
    in the case \cref{item-case-glob-quot},
    by \textcite[Theorem~4.3.1]{alper-2014-adequate},
    the morphism $Y \to X$ is affine,
    so the base change along any étale morphism
    $X' \to X$ with~$X'$ affine gives the desired form.

    Now let $\mathcal{X} = U / \mathrm{GL}_n$ for an affine scheme~$U$,
    let $T \subset \mathrm{GL}_n$ be a maximal torus,
    and let~$\lambda, \mu \colon \mathbb{G}_\mathrm{m} \to T$ be cocharacters.
    By \textcite[Theorem~1.4.7]{halpern-leistner-instability}, we have
    \begin{equation*}
        \mathcal{X}_\lambda \simeq
        U_\lambda / L_\lambda \ ,
        \qquad
        \mathcal{X}_\lambda^+ \simeq
        U_\lambda^+ / P_\lambda \ ,
    \end{equation*}
    where $U_\lambda, U_\lambda^+ \subset U$
    are the fixed and attracting loci of~$\lambda$, respectively, described in
    \textcite[\S 1.4.7]{drinfeld-gaitsgory-2014-braden},
    and $L_\lambda, P_\lambda \subset \mathrm{GL}_n$
    are the Levi and parabolic subgroups given by~$\lambda$.
    It follows that we have
    \begin{equation*}
        \calSt_{\lambda, \mu} \simeq P_\lambda \backslash Z / P_\mu \ ,
        \qquad
        Z = \{ (x, g, y) \in U_\lambda^+ \times \mathrm{GL}_n \times U_\mu^+ \mid x = g y \} \ ,
    \end{equation*}
    and for each
    $w \in |P_\lambda \backslash \mathrm{GL}_n / P_\mu| \simeq
    W_\lambda \backslash W / W_\mu$,
    where $W_\lambda, W_\mu \subset W$ are the stabilizers of~$\lambda, \mu$,
    we have the locally closed subset
    \begin{equation*}
        \calSt_{\lambda, \mu}^w \simeq P_\lambda \backslash Z_w / P_\mu \ ,
        \qquad
        Z_w = \{ (x, g, y) \in Z \mid g \in P_\lambda w P_\mu \} \ .
    \end{equation*}
    The Bruhat decomposition
    $\mathrm{GL}_n = \bigsqcup_w P_\lambda w P_\mu$
    implies that we have a stratification
    \begin{equation*}
        \calSt_{\lambda, \mu} =
        \bigsqcup_{w \in W_\lambda \backslash W / W_\mu}
        \calSt_{\lambda, \mu}^w \ ,
    \end{equation*}
    with the order given by any total order refining the Bruhat order on
    $W_\lambda \backslash W / W_\mu$.
    We will show that the connected components of
    $\calSt_{\lambda, \mu}^w$
    give the desired stratification,
    so that components with the same~$w$ can be ordered arbitrarily.

    Indeed, now write $U = \operatorname{Spec} A$,
    and let $w \in W_\lambda \backslash W / W_\mu$.
    The $\mathbb{G}_\mathrm{m}^2$-action on~$U$
    given by $(\lambda, w \mu)$
    equips~$A$ with a $\mathbb{Z}^2$-grading
    $A = \bigoplus_{\smash{(a, b) \in \mathbb{Z}^2}} A_{a, b}$,
    and by the description in
    \cite[\S 1.4.7]{drinfeld-gaitsgory-2014-braden}, we have
    \begin{equation*}
        U_\lambda^+ = \operatorname{Spec}
        \frac{A}{( A_{a, b} : a < 0 )} \ ,
        \qquad
        U_{w \mu}^+ = \operatorname{Spec}
        \frac{A}{( A_{a, b} : b < 0 )} \ ,
    \end{equation*}
    as closed subschemes of~$U$,
    where the denominators are ideals generated by the graded parts.
    We may then identify
    \begin{equation*}
        \calSt_{\lambda, \mu}^w
        \simeq \frac{U_\lambda^+ \cap U_{w \mu}^+}{P_\lambda \cap w P_\mu w^{-1}} \ ,
        \qquad
        U_\lambda^+ \cap U_{w \mu}^+ \simeq \operatorname{Spec}
        \frac{A}{( A_{a, b} : a < 0 \text{ or } b < 0 )}
        \simeq (U_\lambda^+)_{w \mu}^+ \ ,
    \end{equation*}
    where we take scheme-theoretic intersections.
    As a consequence,
    $\calSt_{\lambda, \mu}^w \simeq
    ((U / \mathrm{GL}_n)_\lambda^+)_{w \mu}^+$.

    The proof is now complete in the case \cref{item-case-glob-quot},
    with the order of the stratification
    induced by the Bruhat order for $\mathrm{GL}_n$ as above.

    In the case \cref{item-case-gms},
    it remains to show that the order of the stratification
    when $\mathcal{X} = U / \mathrm{GL}_n$
    is compatible with the order given by any norm on $\mathrm{CL} (\mathcal{X})$.
    It is enough to check this on an étale neighbourhood
    of any closed point $x \in \mathcal{X}$.
    Let $\mathrm{Aut} (x)$ denote the stabilizer group of~$x$.

    Applying \cref{lem-local-str}
    implies that we may further replace $U / \mathrm{GL}_n$
    by $V / G$ defined over a base scheme~$S'$,
    where $V$ is affine,
    $G$ is a linearly reductive group scheme,
    and there exists $s \in S'$ with $G_s \simeq \mathrm{Aut} (x)_{\kappa'}$
    and a point $0 \in V_s$ fixed by $G_s$,
    where $\kappa' = \kappa (s)$ is a finite extension of~$\kappa (x)$.
    By \cite[Theorem~9.9]{alper-hall-rydh-2026-etale},
    there is a well-defined unit component $G^\circ \subset G$ which is reductive,
    and $G / G^\circ \to S'$ is finite étale.
    Taking a further étale neighbourhood,
    by \cite[Lemma~5.1.3]{conrad-2014-reductive},
    we may assume that $G^\circ$ is split reductive and
    $G / G^\circ \simeq \uppi_0 (G)$ is a constant finite group.
    Let $T \subset G^\circ$ be a split maximal torus.

    Now, a theorem of \textcite[Theorem]{luna-1975-adherences},
    strengthened by \textcite{ibanez-nunez-finiteness}
    to allow more general bases,
    implies that for any cocharacter
    $\lambda \colon \mathbb{G}_\mathrm{m} \to T$,
    the morphism
    $V_\lambda \git L_\lambda \to V \git G$
    is finite, where $V_\lambda \subset V$ is the $\lambda$-fixed locus.
    Therefore, shrinking~$V$,
    we may assume that $V_\lambda / L_\lambda$ is connected for all~$\lambda$.
    Indeed, we may remove from~$V \git G$ the images of
    $(V_\lambda \setminus V_\lambda^\circ) \git L_\lambda$
    for all~$\lambda$, where
    $V_\lambda^\circ / L_\lambda \subset V_\lambda / L_\lambda$
    is the connected component containing~$0$,
    and the cited result implies that such images are closed.
    Only finitely many~$\lambda$ are needed for this process.

    We now have $\mathrm{CL} (V / G) \simeq \Lambda_T / W$,
    where $W = \mathrm{N}_G (T) / \mathrm{Z}_G (T)$ is the Weyl group,
    as a constant finite group on~$S'$.
    Crucially, the original quadratic norm on~$\mathrm{CL} (\mathcal{X})$
    now induces a $W$-invariant inner product on~$\Lambda_T$.
    For cocharacters $\lambda, \mu \colon \mathbb{G}_\mathrm{m} \to T$,
    the strata of $\calSt_{\lambda, \mu} (V / G)$
    are given by preimages of Bruhat strata in
    $P_\lambda \backslash G / P_\mu$,
    whose order can now be checked on geometric fibres,
    and the result follows from \cref{lem-bruhat-ineq}.
\end{proof}

\begin{lemma}
    \label{lem-local-str}
    Let~$\mathcal{X}$ be a stack as in
    \cref{theorem:stratification}~\cref{item-case-gms}.
    Then for any closed point $x \in \mathcal{X}$
    with residue field $\kappa = \kappa (x)$,
    there exists a cartesian diagram
    \begin{equation*}
        \begin{tikzcd}
            \llap{$V / G \simeq {}$}
            \mathcal{X}'
            \ar[r, "f"]
            \ar[d, "\pi'"']
            \ar[dr, phantom, pos=.2, "\ulcorner"]
            &
            \mathcal{X}
            \ar[d, "\pi"]
            \\
            X'
            \ar[r, "\bar{f}"]
            &
            X \rlap{ ,}
        \end{tikzcd}
    \end{equation*}
    where $f$ and $\bar{f}$ are étale,
    $X'$ and~$V$ are affine schemes,
    $G$ is a linearly reductive group scheme over~$X'$,
    and there is a point $0 \in V$ with residue field
    $\kappa' = \kappa (0)$ which is a finite extension of~$\kappa$,
    fixed by the action of $G_{\bar{0}}$,
    where $\bar{0} = \pi' (0)$ also has residue field~$\kappa'$,
    and $f (0) = x$ inducing an isomorphism on stabilizer groups
    $G_{\bar{0}} \simeq \mathrm{Aut} (x)$.
\end{lemma}

\begin{proof}
    By \textcite[Theorem~6.1]{alper-hall-rydh-2026-etale},
    we may assume $\mathcal{X} \simeq U / \mathrm{GL}_n$
    for an affine scheme~$U$.
    The lemma now follows from applying
    \cite[Theorem~10.4]{alper-hall-rydh-2026-etale}
    to this situation.
\end{proof}

\begin{lemma}
    \label{lem-bruhat-ineq}
    Let~$k$ be an algebraically closed field of characteristic~$0$,
    $G$ a linearly reductive group over~$k$
    with a maximal torus~$T$,
    and let $\lambda, \mu \colon \mathbb{G}_\mathrm{m} \to T$ be cocharacters.
    Define the Bruhat partial order $\preceq$ on the double quotient
    $|P_\lambda \backslash G / P_\mu|$
    to be given by $w \preceq w'$ if and only if
    $\{ w \} \subset \overline{\{ w' \}}$.

    Then, for any $W$-invariant inner product $\langle -, - \rangle$
    on the cocharacter lattice~$\Lambda_T$ of\/~$T$,
    and for any $w, w' \in |P_\lambda \backslash G / P_\mu|$,
    we have that
    \begin{equation*}
        w \prec w'
        \quad \text{implies} \quad
        \langle \lambda, w \mu \rangle > \langle \lambda, w' \mu \rangle \ .
    \end{equation*}
\end{lemma}

\begin{proof}
    We first assume that~$G$ is connected.
    We may assume that $\lambda, \mu$ are dominant,
    and identify $|P_\lambda \backslash G / P_\mu|$
    with $W_\lambda \backslash W / W_\mu$,
    where $W_\lambda, W_\mu \subset W$ are the stabilizers of $\lambda, \mu$.
    The Bruhat order on the double quotient is induced by that on~$W$,
    and the latter is generated by covering relations
    $w \prec w'$ for $w' = w s_\alpha$,
    where $s_\alpha$ is the reflection along a positive root~$\alpha$,
    and $\ell (w') = \ell (w) + 1$.
    It suffices to prove the lemma for such $w, w'$
    when they give different elements of
    $W_\lambda \backslash W / W_\mu$.
    By \textcite[Proposition~5.7]{humphreys-1990-coxeter},
    the length condition implies that $w \alpha$ is a positive root.
    Thus
    \begin{align*}
        \langle \lambda, w \mu \rangle - \langle \lambda, w' \mu \rangle
        & = \langle \lambda, w (\mu - s_\alpha \mu) \rangle
        \\
        & = \Bigl< \lambda, w \cdot 2 \frac{\langle \mu, \alpha \rangle}{\langle \alpha, \alpha \rangle} \alpha \Bigr>
        \\
        & =
        \frac{2}{\langle \alpha, \alpha \rangle}
        \langle \lambda, w \alpha \rangle
        \langle \mu, \alpha \rangle > 0 \ .
    \end{align*}
    Here, we see roots as cocharacters
    (with $\mathbb{R}$-coefficients)
    using the inner product.
    The last inequality is strict,
    because we assumed that $W_\lambda w' W_\mu \neq W_\lambda w W_\mu$;
    if $\langle \mu, \alpha \rangle = 0$,
    then $w' \mu = w \mu$ and $w' W_\mu = w W_\mu$;
    if $\langle \lambda, w \alpha \rangle = 0$,
    then $\langle w^{-1} \lambda, \alpha \rangle = 0$,
    so $w^{-1} \lambda = w'^{-1} \lambda$
    and $W_\lambda w' = W_\lambda w$.

    Now assume that~$G$ is disconnected.
    Let $G^\circ \subset G$ be the unit component,
    and let $W^\circ \subset W$ be the Weyl group of~$G^\circ$,
    which is a normal subgroup.
    Since the $W$-action on~$\Lambda_T$ preserves the root hyperplanes,
    each $W^\circ$-coset in $W$
    contains a unique element preserving the dominant chamber.
    Such elements form a subgroup
    $W^* = (\mathrm{N}_G (T) \cap \mathrm{N}_G (B)) / \mathrm{Z}_G (T) \subset W$,
    and we have a semidirect product
    $W = W^\circ \rtimes W^*$.
    Since $\lambda, \mu$ are dominant, one can deduce that
    $W_\lambda = W_\lambda^\circ \rtimes W_\lambda^*$
    with $W_\lambda^\circ \subset W^\circ$ and $W_\lambda^* \subset W^*$
    the stabilizers of~$\lambda$, and similarly for~$\mu$.
    We thus have a partition
    \begin{equation*}
        W_\lambda \backslash W / W_\mu \simeq
        \coprod_{w \in W_\lambda^* \backslash W^* / W_{\mu}^*}
        \frac{W_\lambda^\circ \backslash W^\circ / W_{w \mu}^\circ}
        {W_\lambda^* \cap W_{w \mu}^*}
    \end{equation*}
    as a partially ordered set,
    where we use the Bruhat order on
    $W_\lambda^\circ \backslash W^\circ / W_{w \mu}^\circ$
    as in the connected case,
    which descends to the quotient by the conjugation action of
    $W_\lambda^* \cap W_{w \mu}^*$,
    and the lemma follows from the connected case.
\end{proof}

\noindent
Finally, we record a corollary of \cref{theorem:stratification}
which seems to be well-known,
but we could not find a direct reference for it,
so we include a simple proof here. We refer to \cite[\S3.3]{chriss1997representation} for a closely related variety.

\begin{corollary}
    \label{cor-irred-comp}
    Let~$k$ be a field,
    let~$G$ be a reductive group over~$k$
    with a split maximal torus $T \subset G$,
    and let $\lambda, \mu \colon \mathbb{G}_\mathrm{m} \to T$ be cocharacters.
    Then the Steinberg stacks for adjoint quotients
    \begin{equation*}
        \calSt (\mathfrak{g} / G)_{\lambda, \mu} \simeq
        \mathfrak{p}_\lambda / P_\lambda
        \underset{\mathfrak{g} / G}{\times}
        \mathfrak{p}_\mu / P_\mu
        \qquad\text{and}\qquad
        \calSt (G / G)_{\lambda, \mu} \simeq
        P_\lambda / P_\lambda
        \underset{G / G}{\times}
        P_\mu / P_\mu
    \end{equation*}
    are equidimensional of dimension~$0$,
    and have irreducible components indexed by the double quotient
    $W_\lambda \backslash W / W_\mu$,
    where $W = \mathrm{N}_G (T) / T$ is the Weyl group,
    and $W_\lambda, W_\mu \subset W$
    are the stabilizers of~$\lambda, \mu$.
    Each irreducible component is
    the closure of a stratum given by
    \cref{theorem:stratification}.
\end{corollary}

\begin{proof}
    We write down the proof for~$\mathfrak{g} / G$,
    and the case of~$G / G$ is similar.

    Let~$Z$ and~$Z_w$ be as in the proof of \cref{theorem:stratification},
    that is
    \begin{align*}
        Z & = \{ (x, g, y) \in \mathfrak{p}_\lambda \times G \times \mathfrak{p}_\mu \mid x = \mathrm{Ad}_g \, y \} \ ,
        \\
        Z_w & = \{ (x, g, y) \in Z \mid g \in P_\lambda w P_\mu \} \ ,
    \end{align*}
    where $w \in W_\lambda \backslash W / W_\mu$,
    so that $\calSt_{\lambda, \mu} \simeq P_\lambda \backslash Z / P_\mu$,
    and we have a stratification
    $Z = \bigsqcup_w Z_w$.
    But the theorem implies that
    $P_\lambda \backslash Z_w / P_\mu \simeq
    (\mathfrak{p}_\lambda \cap \mathfrak{p}_{w \mu}) / (P_\lambda \cap P_{w \mu})$,
    which is a $0$-dimensional smooth stack,
    so $Z_w$ is smooth of dimension
    $\dim P_\lambda + \dim P_\mu$,
    and it is connected since $P_\lambda$ and $P_\mu$ are connected.
    In particular, all~$Z_w$ are smooth of the same dimension and are connected,
    which implies the corollary.
\end{proof}

\noindent
This also implies that the derived Steinberg stacks in this example,
defined as the corresponding derived fibre products, are classical.
This is because they are quasi-smooth with virtual dimension~$0$,
which coincides with the dimension of their classical truncations.

%% file: mackey.tex
Throughout this section,
let~$S$ be a noetherian algebraic space,
and let~$\mathcal{X}$ be an algebraic stack of finite type over~$S$
with affine stabilizers and separated diagonal,
and with quasi-compact graded points as in \cref{def-qc-graded-points}.

Let $\mathcal{D}$ be a six-functor formalism satisfying
\cref{assumption-six-functors}.

\subsection{Induction and restriction functors}

Recall from \cref{def-hall-cat}
the Hall category
$\mathsf{Hall} (\mathcal{X})$ of~$\mathcal{X}$.
For a morphism
$\alpha \overset{\sigma}{\to} \beta$
in the Hall category of~$\mathcal{X}$,
recall from \cref{def-ind-res} the induction and restriction functors
\begin{alignat*}{5}
    \mathrm{Ind}_\sigma
    & =
    (p_\sigma)_! \, q_\sigma^*
    && \ , \quad
    & \mathrm{Ind}'_\sigma
    & =
    (p_\sigma)_* \, q_\sigma^!
    && \ \colon \quad
    & \mathcal{D} (\mathcal{X}_\beta)
    & \longrightarrow \mathcal{D} (\mathcal{X}_\alpha) \ ,
    \\
    \mathrm{Res}_\sigma
    & =
    (q_\sigma)_! \, p_\sigma^*
    && \ , \quad
    & \mathrm{Res}'_\sigma
    & =
    (q_\sigma)_* \, p_\sigma^!
    && \ \colon \quad
    & \mathcal{D} (\mathcal{X}_\alpha)
    & \longrightarrow \mathcal{D} (\mathcal{X}_\beta) \ .
\end{alignat*}
Recall also that when~$\mathcal{X}$ is smooth and $\Theta$-reductive,
we have an isomorphism
\begin{equation}
    \label{eq-ind-equals-ind-prime}
    \mathrm{Ind}'_\sigma \simeq \mathrm{Ind}_\sigma \langle \dim q_\sigma \rangle \ ,
\end{equation}
and that when~$\mathcal{D}$ satisfies Braden's theorem for~$\mathcal{X}$ as in
\Cref{def-braden-theorem},
we have an isomorphism
\begin{equation}
    \label{eq-res-equals-res-prime}
    \mathrm{Res}'_\sigma \simeq \mathrm{Res}_{-\sigma} \ .
\end{equation}

\begin{proposition}[Associativity]
    \label{prop-ind-res-associativity}
    Let
    $\alpha \overset{\sigma}{\to} \beta \overset{\tau}{\to} \gamma$
    be morphisms in $\mathsf{Hall} (\mathcal{X})$.
    Then we have isomorphisms
    \begin{alignat*}{2}
        \mathrm{Ind}_\sigma \circ \mathrm{Ind}_\tau
        & \simeq
        \mathrm{Ind}_{\sigma \uparrow \tau} \ ,
        & \qquad
        \mathrm{Ind}'_\sigma \circ \mathrm{Ind}'_\tau
        & \simeq
        \mathrm{Ind}'_{\smash{\sigma \uparrow \tau}} \ ,
        \\
        \mathrm{Res}_\tau \circ \mathrm{Res}_\sigma
        & \simeq
        \mathrm{Res}_{\sigma \uparrow \tau} \ ,
        & \qquad
        \mathrm{Res}'_\tau \circ \mathrm{Res}'_\sigma
        & \simeq
        \mathrm{Res}'_{\smash{\sigma \uparrow \tau}} \ .
    \end{alignat*}
\end{proposition}

\begin{proof}
    These are direct consequences of
    \cref{thm-associativity}
    by applying base change formulas.
\end{proof}

\begin{proposition}[Verdier dual]
    \label{prop-ind-res-verdier}
    Let
    $\alpha \overset{\sigma}{\to} \beta$
    be a morphism in $\mathsf{Hall} (\mathcal{X})$.
    Then we have isomorphisms
    \begin{equation*}
        \mathbb{D} \, \mathrm{Ind}_\sigma
        \simeq
        \mathrm{Ind}'_\sigma \, \mathbb{D} \ ,
        \qquad
        \mathbb{D} \, \mathrm{Res}_\sigma
        \simeq
        \mathrm{Res}'_\sigma \, \mathbb{D} \ .
    \end{equation*}
\end{proposition}

\begin{proof}
    These follow from that Verdier duality intertwines the
    $*$-push and the $!$-push,
    and the $*$-pull and the $!$-pull,
    as mentioned in \cref{assumption-six-functors}.
\end{proof}

\begin{proposition}[Adjunction sequence]
    \label{prop-adjunction-seq}
    Assume that $\mathcal{X}$ is smooth and $\Theta$-reductive,
    and that~$\mathcal{D}$ satisfies Braden's theorem for~$\mathcal{X}$ as in
    \Cref{def-braden-theorem}.
    Then for each morphism
    $\alpha \overset{\sigma}{\to} \beta$
    in $\mathsf{Hall}(\mathcal{X})$,
    we have an infinite sequence of adjunctions
    \begin{equation*}
        \cdots
        \dashv
        \mathrm{Ind}_{\sigma}
        \dashv
        \mathrm{Res}_{-\sigma}
        \dashv
        \mathrm{Ind}_{-\sigma} \langle d_- \rangle
        \dashv
        \mathrm{Res}_{\sigma} \langle -d_- \rangle
        \dashv
        \mathrm{Ind}_{\sigma} \langle d_- + d_+ \rangle
        \dashv
        \mathrm{Res}_{-\sigma} \langle -d_- - d_+ \rangle
        \dashv
        \cdots \ ,
    \end{equation*}
    where $d_{\pm} = \dim q_{\pm \sigma}$.
\end{proposition}

\begin{proof}
    This is a combination of the adjunctions
    $\mathrm{Ind}_{\sigma} \dashv \mathrm{Res}'_{\sigma}$
    and $\mathrm{Res}_{\sigma} \dashv \mathrm{Ind}'_{\sigma}$
    with the isomorphisms
    \cref{eq-ind-equals-ind-prime,eq-res-equals-res-prime}.
\end{proof}

\begin{proposition}[Induction preserves purity]
    \label{prop-ind-preserves-purity}
    Assume we are in the situation of
    mixed Hodge modules as in
    \cref{eg-mhm},
    and assume that $\mathcal{X}$ is smooth and $\Theta$-reductive.
    Let $\alpha \overset{\sigma}{\to} \beta$ be a morphism in $\mathsf{Hall}(\mathcal{X})$.

    Then the functors $\mathrm{Ind}_{\sigma}$ and $\mathrm{Ind}'_{\sigma}$
    preserve pure mixed Hodge modules of each weight.
\end{proposition}

\begin{proof}
    Under the assumptions of the proposition,
    the induction functors are compositions of
    smooth pullbacks and proper pushforwards,
    which both preserve purity and weight.
\end{proof}

\begin{proposition}[Restriction preserves purity]
    \label{prop-res-preserves-purity}
    Assume we are in the situation of mixed Hodge modules
    as in \cref{eg-mhm}.
    Let $\alpha \overset{\sigma}{\to} \beta$ be a morphism in $\mathsf{Hall}(\mathcal{X})$.

    Then the functors
    $\mathrm{Res}_{\sigma}$ and $\mathrm{Res}'_{\sigma}$
    preserve pure mixed Hodge modules of each weight.
\end{proposition}

\begin{proof}
    By \textcite[Proposition~A.1 and Remark~A.3]{kinjo2024cohomological},
    Braden's theorem in the sense of \cref{def-braden-theorem} holds.

    Let $E \in \mathcal{D} (\mathcal{X}_\alpha)_w$ be a pure object of weight~$w$.
    Then $\mathrm{Res}_{\pm \sigma} \, E \in \mathcal{D} (\mathcal{X}_\beta)_{\leq w}$
    and $\mathrm{Res}'_{\pm \sigma} \, E \in \mathcal{D} (\mathcal{X}_\beta)_{\geq w}$.
    But since
    $\mathrm{Res}_\sigma \, E \simeq \mathrm{Res}'_{-\sigma} \, E$
    and
    $\mathrm{Res}_{-\sigma} \, E \simeq \mathrm{Res}'_\sigma \, E$,
    they must all lie in $\mathcal{D} (\mathcal{X}_\beta)_{w}$.
\end{proof}

\begin{proposition}[Restrictions of constant/dualizing objects]
    Assume that~$\mathcal{X}$ is smooth,
    and let $\alpha\xrightarrow{\sigma}\beta$
    be a morphism in $\mathsf{Hall} (\mathcal{X})$. Then we have
    \[
        \mathrm{Res}_{\sigma} \, \mathbb{1}_{\mathcal{X}_{\alpha}}
        \simeq \mathbb{1}_{\mathcal{X}_{\beta}} \langle -{\dim q_\sigma} \rangle \ ,
        \qquad
        \mathrm{Res}'_{\sigma} \, \mathbb{D1}_{\mathcal{X}_{\alpha}}
        \simeq (\mathbb{D1}_{\mathcal{X}_{\beta}})
        \langle \dim q_\sigma \rangle \ .
    \]
\end{proposition}

\begin{proof}
    By \cref{prop-ind-res-verdier},
    it is enough to show the first formula. We have
    \begin{align*}
        \mathrm{Res}_{\sigma} \, \mathbb{1}_{\mathcal{X}_{\alpha}}
        \simeq (q_\sigma)_! \, p_{\sigma}^* \, \mathbb{1}_{\mathcal{X}_{\alpha}}
        \simeq (q_\sigma)_! \, \mathbb{1}_{\mathcal{X}_{\sigma}^+}
        \simeq (q_\sigma)_! \, q_{\sigma}^* \, \mathbb{1}_{\mathcal{X}_{\beta}}
        \simeq \mathbb{1}_{\mathcal{X}_{\beta}} \langle -{\dim q_{\sigma}} \rangle \ ,
    \end{align*}
    where the last follows from the fact that~$q_\sigma$
    is smooth and an $\mathbb{A}^1$-deformation retract,
    so we may apply \cref{lem-a1-contract},
    and take left adjoints of the statement of the lemma.
\end{proof}

\subsection{The Mackey formula}

Let~$\mathcal{X}$ be a stack over a base~$S$ as before,
and let $\lambda, \mu \in \mathrm{CL} (\mathcal{X})$ be two elements.
We wish to study the composition of the functors
\begin{equation}
    \label{eq-st-functor}
    \mathrm{St}_{\lambda, \mu}
    = \mathrm{Res}_\mu \circ \mathrm{Ind}_\lambda \ ,
    \quad
    \mathrm{St}'_{\smash{\lambda, \mu}}
    = \mathrm{Res}'_\mu \circ \mathrm{Ind}'_\lambda
    \colon \quad
    \mathcal{D} (\mathcal{X}_\lambda)
    \longrightarrow \mathcal{D} (\mathcal{X}_\mu) \ ,
\end{equation}
which we call the \emph{Steinberg functor}.
We have the adjunction
$\mathrm{St}_{\lambda, \mu}\dashv \mathrm{St}'_{\smash{\mu, \lambda}}$.

More generally, for morphisms
$\beta \overset{\sigma}{\gets} \alpha \overset{\tau}{\to} \gamma$
in the category $\mathsf{Hall} (\mathcal{X})$,
we may consider the composition
\begin{equation}
    \label{eq-st-functor-cones}
    \mathrm{St}_{\sigma, \tau}
    = \mathrm{Res}_\tau \circ \mathrm{Ind}_\sigma \ ,
    \quad
    \mathrm{St}'_{\sigma, \tau}
    = \mathrm{Res}'_\tau \circ \mathrm{Ind}'_\sigma
    \colon \quad
    \mathcal{D} (\mathcal{X}_\beta)
    \longrightarrow \mathcal{D} (\mathcal{X}_\gamma) \ ,
\end{equation}
which agree with~\cref{eq-st-functor}
with~$\mathcal{X}_\alpha$ in place of~$\mathcal{X}$,
and with $\lambda$ and~$\mu$ chosen generically in~$\sigma$ and~$\tau$, respectively.
Therefore, we work with~\cref{eq-st-functor} below
without losing generality.

By the definition of the Steinberg stack
$\calSt_{\lambda, \mu} = \calSt (\mathcal{X})_{\lambda, \mu}$,
we have a diagram
\begin{equation*}
    \begin{tikzcd}[column sep={4em, between origins}, row sep={2.5em, between origins}]
        &&
        \calSt_{\lambda, \mu}
        \ar[dl] \ar[dr]
        \ar[ddll, bend right=20, shift right, "\phi_{\lambda, \mu}"']
        \ar[ddrr, bend left=20, shift left, "\phi_{\mu, \lambda}"]
        \ar[dd, phantom, pos=.2, "\ulcorner" rotate=-45]
        &&
        \\
        &
        \mathcal{X}_\lambda^+
        \ar[dl, "q_\lambda"] \ar[dr, "p_\lambda"']
        &&
        \mathcal{X}_\mu^+
        \ar[dl, "p_\mu"] \ar[dr, "q_\mu"']
        \\
        \mathcal{X}_\lambda && \mathcal{X} && \mathcal{X}_\mu \rlap{ ,}
    \end{tikzcd}
\end{equation*}
where the square is cartesian, so that by the base change formula, we have
\begin{equation*}
    \mathrm{St}_{\lambda, \mu} \simeq
    (\phi_{\mu, \lambda})_! \,
    \phi_{\lambda, \mu}^* \ ,
    \qquad
    \mathrm{St}'_{\smash{\lambda, \mu}} \simeq
    (\phi_{\mu, \lambda})_* \,
    \phi_{\lambda, \mu}^! \ .
\end{equation*}

When~$\mathcal{X}$ satisfies \cref{assumption-stratification},
we use the stratification of $\calSt_{\lambda, \mu}$
given by \cref{theorem:stratification}
to produce a filtration of the functors
$\mathrm{St}_{\lambda, \mu}$ and~$\mathrm{St}'_{\smash{\lambda, \mu}}$.
Namely, let~$I$ be the set of relative positions of~$\lambda$ and~$\mu$,
with the total order $\leq$ as given by the theorem.
For each $i \in I$, we have the locally closed stratum
\begin{equation*}
    j_i \colon
    \calSt_{\lambda, \mu}^i =
    (\mathcal{X}_\lambda^+)_{\mu_i}^+ \simeq
    (\mathcal{X}_\mu^+)_{\lambda_i}^+
    \longhookrightarrow \calSt_{\lambda, \mu} \ ,
\end{equation*}
where $\lambda_i, \mu_i$ are lifts as in \cref{def-relative-position}.
We let $\phi_{\smash{\lambda, \mu}}^i = \phi_{\lambda, \mu} \circ j_i$
and $\phi_{\smash{\mu, \lambda}}^i = \phi_{\mu, \lambda} \circ j_i$,
and define the partial Steinberg functors
\begin{equation*}
    \mathrm{St}^i_{\smash{\lambda, \mu}} =
    (\phi_{\smash{\mu, \lambda}}^i)_! \,
    (\phi_{\smash{\lambda, \mu}}^i)^* \ ,
    \quad
    \mathrm{St}'^i_{\smash{\lambda, \mu}} =
    (\phi_{\smash{\mu, \lambda}}^i)_* \,
    (\phi_{\smash{\lambda, \mu}}^i)^!
    \colon \quad
    \mathcal{D} (\mathcal{X}_\lambda)
    \longrightarrow \mathcal{D} (\mathcal{X}_\mu) \ .
\end{equation*}
For each $i \in I$, we also consider the open substack
\begin{equation*}
    j_{\geq i} \colon
    \calSt_{\smash{\lambda, \mu}}^{\geq i} = \bigcup_{i' \geq i} \calSt_{\smash{\lambda, \mu}}^{\smash{i'}}
    \longhookrightarrow \calSt_{\lambda, \mu} \ ,
\end{equation*}
and define the partial Steinberg functors
$\mathrm{St}^{\geq i}_{\smash{\lambda, \mu}}$ and
$\mathrm{St}'^{\geq i}_{\smash{\lambda, \mu}}$
similarly.
We have localization exact triangles
\begin{equation*}
    (j_{> i})_! \, j_{> i}^* \longrightarrow
    (j_{\geq i})_! \, j_{\geq i}^* \longrightarrow
    (j_i)_! \, j_i^* \ ,
    \qquad
    (j_i)_* \, j_i^! \longrightarrow
    (j_{\geq i})_* \, j_{\geq i}^! \longrightarrow
    (j_{> i})_* \, j_{> i}^! \ ,
\end{equation*}
where the second one is given by
\cref{assumption-six-functors}~\cref{item-assumption-localization},
and the first one is obtained as its left adjoint.
They give exact triangles
\begin{equation*}
    \mathrm{St}_{\smash{\lambda, \mu}}^{> i} \longrightarrow
    \mathrm{St}_{\smash{\lambda, \mu}}^{\geq i} \longrightarrow
    \mathrm{St}_{\smash{\lambda, \mu}}^{i} \ ,
    \qquad
    \mathrm{St}_{\smash{\lambda, \mu}}'^{i} \longrightarrow
    \mathrm{St}_{\smash{\lambda, \mu}}'^{\geq i} \longrightarrow
    \mathrm{St}_{\smash{\lambda, \mu}}'^{> i} \ .
\end{equation*}
This proves the following proposition.

\begin{proposition}
    \label{prop-st-filtrations}
    Let~$\mathcal{X}$ satisfy \cref{assumption-stratification}.
    Then $\mathrm{St}_{\smash{\lambda, \mu}}$
    admits a filtration indexed by $(I, \geq)$
    with associated pieces~$\mathrm{St}_{\smash{\lambda, \mu}}^i$,
    and $\mathrm{St}'_{\smash{\lambda, \mu}}$
    admits a filtration indexed by $(I, \leq)$
    with associated pieces~$\mathrm{St}_{\smash{\lambda, \mu}}'^i$.
    \qed
\end{proposition}

\noindent
We now study the functors
$\mathrm{St}_{\lambda, \mu}^i$ and~$\mathrm{St}_{\smash{\lambda, \mu}}'^i$.
Consider the commutative diagram
\begin{equation}
    \label{eq-affine-bundle-diagram}
    \begin{tikzcd}[column sep={5em, between origins}, row sep={3em, between origins}]
        &&
        \calSt_{\lambda, \mu}^i
        \ar[dl, "r^i_\lambda"]
        \ar[dr, "r^i_\mu"']
        \ar[ddll, bend right=20, shift right, "\phi^i_{\lambda, \mu}"']
        \ar[ddrr, bend left=20, shift left, "\phi^i_{\mu, \lambda}"]
        &&
        \\
        &
        (\mathcal{X}_\lambda)_{\mu_i}^+
        \ar[dl, "p^i_\mu"]
        \ar[dr, "q^i_\mu"']
        &&
        (\mathcal{X}_\mu)_{\lambda_i}^+
        \ar[dl, "q^i_\lambda"]
        \ar[dr, "p^i_\lambda"']
        \\
        \mathcal{X}_\lambda
        && \mathcal{X}_{\lambda, \mu}^i
        && \mathcal{X}_\mu \rlap{ ,}
    \end{tikzcd}
\end{equation}
where we define
$\mathcal{X}_{\lambda, \mu}^i =
(\mathcal{X}_\lambda)_{\mu_i} \simeq (\mathcal{X}_\mu)_{\lambda_i}$,
and we study the induced morphism
\begin{equation*}
    h_{\lambda, \mu}^i \colon
    \mathrm{St}_{\lambda, \mu}^i \longrightarrow
    (\mathcal{X}_\lambda)_{\mu_i}^+
    \underset{\mathcal{X}_{\lambda, \mu}^i}{\times}
    (\mathcal{X}_\mu)_{\lambda_i}^+ \ .
\end{equation*}

\begin{lemma}
    \label{lem-affine-bundle}
    Let~$\mathcal{X}$ be a stack satisfying \cref{assumption-stratification},
    and assume that~$\mathcal{X}$ is smooth over~$S$.
    Consider the morphism
    \begin{equation}
        \label{eq-affine-bundle-lemma}
        h \colon
        \calFilt^2 (\mathcal{X})
        \longrightarrow
        \calGrad (\calFilt (\mathcal{X}))
        \underset{\calGrad^2 (\mathcal{X})}{\times}
        \calFilt (\calGrad (\mathcal{X}))
    \end{equation}
    induced by the inclusions
    $\mathrm{B} \mathbb{G}_\mathrm{m} \times \mathbb{A}^1 / \mathbb{G}_\mathrm{m} \to
    \mathbb{A}^2 / \mathbb{G}_\mathrm{m}^2
    \gets \mathbb{A}^1 / \mathbb{G}_\mathrm{m} \times \mathrm{B} \mathbb{G}_\mathrm{m}$
    as the coordinate axes.

    Then~$h$ is smooth, and it is cohomologically contractible,
    meaning that the adjunction unit
    $\mathrm{id} \to h_* \, h^*$
    is an isomorphism.
    In particular,
    via \cref{assumption-six-functors}~\cref{item-assumption-smooth-pullback},
    we have an isomorphism 
    \begin{equation}
        h_* \, h^! \simeq \langle \dim h \rangle \ .
    \end{equation}
\end{lemma}

\begin{proof}
    By smooth descent,
    being cohomologically contractible
    can be checked smooth locally,
    which we will do.

    As in the proof of \cref{theorem:stratification},
    étale locally on the good or adequate moduli space,
    we may assume that
    $\mathcal{X} \simeq U / \mathrm{GL}_n$
    for a smooth affine scheme~$U$
    and $n \geq 0$.
    Choosing a connected component on both sides of
    \cref{eq-affine-bundle-lemma}
    corresponding to a face
    $\rho \colon \mathbb{Z}^2 \to \mathrm{CL} (U / \mathrm{GL}_n)$,
    we have induced cocharacters
    $(\lambda, \mu) \colon \mathbb{G}_\mathrm{m}^2 \to T$
    of the maximal torus $T \subset \mathrm{GL}_n$.
    The morphism \cref{eq-affine-bundle-lemma} becomes
    \begin{equation*}
        \frac{U_\lambda^+ \cap U_\mu^+}{P_\lambda \cap P_\mu}
        \longrightarrow
        \frac{U_\lambda \cap U_\mu^+}{L_\lambda \cap P_\mu}
        \underset{\frac{U_\lambda \cap U_\mu}{L_\lambda \cap L_\mu}{}}{\times}
        \frac{U_\lambda^+ \cap U_\mu}{P_\lambda \cap L_\mu} \ .
    \end{equation*}
    Upon a smooth base change, this morphism becomes
    \begin{equation*}
        \frac{U_\lambda^+ \cap U_\mu^+}{N_\lambda \cap N_\mu}
        \longrightarrow
        (U_\lambda \cap U_\mu^+)
        \underset{U_\lambda \cap U_\mu}{\times}
        (U_\lambda^+ \cap U_\mu) \ ,
    \end{equation*}
    where $N_\lambda \subset P_\lambda$ and $N_\mu \subset P_\mu$
    are the unipotent radicals,
    and $N_\lambda \cap N_\mu$ is an affine space over~$S$,
    and is cohomologically contractible by
    \cref{assumption-six-functors}~\cref{item-assumption-a1-invariance}.
    Thus $S \to \mathrm{B} (N_\lambda \cap N_\mu)$
    is cohomologically contractible, hence also
    $U_\lambda^+ \cap U_\mu^+ \to (U_\lambda^+ \cap U_\mu^+) / (N_\lambda \cap N_\mu)$,
    and it is enough to show that the morphism
    \begin{equation}
        \label{eq-affine-bundle-lemma-2}
        U_\lambda^+ \cap U_\mu^+
        \longrightarrow
        (U_\lambda \cap U_\mu^+)
        \underset{U_\lambda \cap U_\mu}{\times}
        (U_\lambda^+ \cap U_\mu)
    \end{equation}
    is smooth and cohomologically contractible.

    In fact, we will show that \cref{eq-affine-bundle-lemma-2}
    is an affine space fibration.
    As in the proof of
    \cref{theorem:stratification},
    we may write
    $U_\lambda^+ \cap U_\mu^+ \simeq \operatorname{Spec} A$
    for a ring~$A$ with an $\mathbb{N}^2$-grading
    $A = \bigoplus_{a, b \geq 0} A_{a, b}$.
    We may now apply a Białynicki-Birula decomposition
    for multiple gradings,
    as in \textcite[Theorem~12.4]{mayeux-2026-magnetism}
    and its proof with the monoid $N = \mathbb{N}^2$,
    to conclude that there is a (non-canonical)
    $\mathbb{G}_\mathrm{m}^2$-equivariant isomorphism
    \begin{equation*}
        U_\lambda^+ \cap U_\mu^+ \simeq V
    \end{equation*}
    for an $\mathbb{N}^2$-graded vector bundle
    $V = \bigoplus_{a, b \geq 0} V_{a, b}
    \to U_\lambda \cap U_\mu$
    with $V_{0, 0} = 0$,
    and $\mathbb{G}_\mathrm{m}^2$ acts on its fibres with the corresponding weights.
    In particular,
    \cref{eq-affine-bundle-lemma-2} becomes the projection
    \begin{equation*}
        V \longrightarrow
        \bigoplus_{a > 0} V_{a, 0} \oplus
        \bigoplus_{b > 0} V_{0, b} \ ,
    \end{equation*}
    and is a vector bundle with fibre
    $\bigoplus_{a, b > 0} V_{a, b}$.
\end{proof}

\begin{corollary}
\label{corollary:twisted-base-change}
    Let~$\mathcal{X}$ be a smooth stack over~$S$
    satisfying \cref{assumption-stratification}.
    Then in the diagram \cref{eq-affine-bundle-diagram},
    we have twisted base change formula
    \begin{equation*}
        (r_\mu^i)_* \, (r_\lambda^i)^!
        \simeq
        (q_\lambda^i)^! \, (q_\mu^i)_* \,
        \langle \dim h_{\lambda, \mu}^i \rangle \ ,
    \end{equation*}
    where $\dim h_{\lambda, \mu}^i =
    \dim r_\lambda^i - \dim q_\lambda^i =
    \dim r_\mu^i - \dim q_\mu^i$.
\end{corollary}

\begin{proof}
    Consider the cartesian square
    \begin{equation*}
        \begin{tikzcd}[row sep=1.2em, column sep=1.6em]
            \calSt_{\lambda, \mu}^i
            \ar[dr, "\smash{h_{\lambda, \mu}^i}"]
            \ar[ddr, bend right=10, shift right, "r_\lambda^i"']
            \ar[drr, bend left=10, shift left, "r_\mu^i"]
            &[-.8em]
            \\[-.8em]
            &
            (\mathcal{X}_\lambda)_{\mu_i}^+
            \underset{\mathcal{X}_{\lambda, \mu}^i}{\times}
            (\mathcal{X}_\mu)_{\lambda_i}^+
            \ar[r, "b"']
            \ar[d, "a"]
            \ar[dr, phantom, start anchor=center, end anchor=center, pos=.4, "\ulcorner"]
            &
            (\mathcal{X}_\mu)_{\lambda_i}^+
            \ar[d, "q_\lambda^i"]
            \\
            &
            (\mathcal{X}_\lambda)_{\mu_i}^+
            \ar[r, "q_\mu^i"]
            &
            \mathcal{X}_{\lambda, \mu}^i \rlap{ .}
        \end{tikzcd}
    \end{equation*}
    We have
    \begin{align*}
        (r_\mu^i)_* \, (r_\lambda^i)^!
        & \simeq
        b_* \, (h_{\lambda, \mu}^i)_* \, (h_{\lambda, \mu}^i)^! \, a^!
        \\
        & \simeq
        b_* \, a^! \langle \dim h_{\lambda, \mu}^i \rangle
        \\
        & \simeq
        (q_\lambda^i)^! \, (q_\mu^i)_* \langle \dim h_{\lambda, \mu}^i \rangle \ ,
    \end{align*}
    where the second step uses \cref{lem-affine-bundle}.
    The expression for
    $\dim h_{\lambda, \mu}^i$
    can be seen from the diagram above.
\end{proof}

\noindent
For~$\mathcal{X}$ as above, we now define the braiding operators
\begin{alignat}{2}
\label{equation:twist}
    T^i_{\lambda, \mu} = \langle -{\dim h_{\lambda, \mu}^i} \rangle \ ,
    \quad
    T^{i \prime}_{\lambda, \mu} = \langle \dim h_{\lambda, \mu}^i \rangle \colon
    \quad
    \mathcal{D} ((\mathcal{X}_{\lambda})_{\mu_i})
    \longsimto
    \mathcal{D} ((\mathcal{X}_{\mu})_{\lambda_i}) \ ,
\end{alignat}
using the canonical identifications
$(\mathcal{X}_{\lambda})_{\mu_i} \simeq \mathcal{X}_{\lambda, \mu}^i
\simeq (\mathcal{X}_{\mu})_{\lambda_i}$.

The integer $\dim h_{\lambda, \mu}^i$ can be alternatively described as follows.
Consider the forgetful morphism
$u_i \colon \mathcal{X}_{\lambda, \mu}^i \to \mathcal{X}$,
and consider the perfect complex
$E^i = u_i^* \, \mathbb{T}_{\mathcal{X}}$
on~$\mathcal{X}_{\lambda, \mu}^i$,
which admits a natural $\mathbb{Z}^2$-grading
by the tautological $\mathrm{B} \mathbb{G}_\mathrm{m}^2$-action
on~$\mathcal{X}_{\lambda, \mu}^i$
via~$\lambda$ and~$\mu_i$.
(Using $\mu$ and~$\lambda_i$ instead will give the same grading,
but with the two factors in~$\mathbb{Z}^2$ swapped.)
We then have
\begin{equation}
    \label{eq-dim-h}
    \dim h_{\lambda, \mu}^i =
    \operatorname{rank} E^i_{> 0, \, > 0} \ .
\end{equation}

\begin{proposition}
    \label{proposition:explicitation_partial_steinberg}
    There exist equivalences of functors
        \begin{align*}
            \mathrm{St}_{\lambda, \mu}^i
            & \simeq
            \mathrm{Ind}_{\lambda_i} \circ T^i_{\lambda, \mu} \circ \mathrm{Res}_{\mu_i} \ ,
            \\
            \mathrm{St}_{\lambda, \mu}'^i
            & \simeq
            \mathrm{Ind}'_{\lambda_i} \circ T^{i \prime}_{\lambda, \mu} \circ \mathrm{Res}'_{\mu_i} \ .
        \end{align*}
\end{proposition}

\begin{proof}
    By adjunction (and swapping the roles of $\lambda$ and~$\mu$),
    it is enough to show the second formula. We have
    \begin{align*}
        \mathrm{St}_{\lambda, \mu}'^i
        & \simeq
        (\phi_{\mu, \lambda}^i)_* \, \phi_{\lambda, \mu}^!
        \\
        & \simeq
        (p_\lambda^i)_* \, (r_\mu^i)_* \, (r_\lambda^i)^! \, (p_\mu^i)^!
        \\
        & \simeq
        (p_\lambda^i)_* \, (q_\lambda^i)^! \, (q_\mu^i)_* \, (p_\mu^i)^! \, \langle \dim h_{\lambda, \mu}^i \rangle
        \\
        & \simeq
        \mathrm{Ind}'_{\lambda_i} \circ T^{i \prime}_{\lambda, \mu} \circ \mathrm{Res}'_{\mu_i} \ ,
    \end{align*}
    where the third step used
    \cref{corollary:twisted-base-change}.
\end{proof}

\begin{theorem}[Mackey formula]
    \label{theorem:mackey-formula}
    Let~$\mathcal{X}$ be a smooth stack over~$S$
    satisfying \cref{assumption-stratification},
    and let $\lambda, \mu \in \mathrm{CL} (\mathcal{X})$.

    Then the compositions
    $\mathrm{Res}_{\mu} \circ \mathrm{Ind}_{\lambda}$
    and
    $\mathrm{Res}'_{\mu} \circ \mathrm{Ind}'_{\lambda}$
    have filtrations with associated graded pieces
    \begin{align*}
        \mathrm{gr}_i (\mathrm{Res}_{\mu} \circ \mathrm{Ind}_{\lambda})
        & \simeq
        \mathrm{Ind}_{\lambda_i} \circ T^i_{\lambda, \mu} \circ \mathrm{Res}_{\mu_i} \ ,
        \\
        \mathrm{gr}_i (\mathrm{Res}'_{\mu} \circ \mathrm{Ind}'_{\lambda})
        & \simeq
        \mathrm{Ind}'_{\lambda_i} \circ T^{i \prime}_{\lambda, \mu} \circ \mathrm{Res}'_{\mu_i} \ .
    \end{align*}
    The two filtrations are indexed by totally ordered sets
    $(I, \geq)$ and $(I, \leq)$, respectively,
    where~$I$ is the set of relative positions of~$\lambda$ and~$\mu$,
    with the total order~$\leq$ given by \cref{theorem:stratification}.

    In particular, for any morphisms
    $\beta \overset{\sigma}{\gets} \alpha \overset{\tau}{\to} \gamma$
    in the category $\mathsf{Hall} (\mathcal{X})$,
    the compositions
    $\mathrm{Res}_{\tau} \circ \mathrm{Ind}_{\sigma}$
    and
    $\mathrm{Res}'_{\tau} \circ \mathrm{Ind}'_{\sigma}$
    have filtrations with associated graded pieces
    \begin{align*}
        \mathrm{gr}_i (\mathrm{Res}_{\tau} \circ \mathrm{Ind}_{\sigma})
        & \simeq
        \mathrm{Ind}_{\sigma_i} \circ T^i_{\sigma, \tau} \circ \mathrm{Res}_{\tau_i} \ ,
        \\
        \mathrm{gr}_i (\mathrm{Res}'_{\tau} \circ \mathrm{Ind}'_{\sigma})
        & \simeq
        \mathrm{Ind}'_{\sigma_i} \circ T^{i \prime}_{\sigma, \tau} \circ \mathrm{Res}'_{\tau_i} \ ,
    \end{align*}
    indexed by $(I, \geq)$ and $(I, \leq)$, respectively,
    for some totally ordered set~$I$ and suitable morphisms
    $\beta \overset{\tau_i}{\to} \rho_i \overset{\sigma_i}{\gets} \gamma$
    in $\mathsf{Hall} (\mathcal{X})$ for $i \in I$.

    If, moreover,
    the finite subcategories $\mathcal{D}_\mathrm{c} (\mathcal{X})$ in
    \cref{assumption-six-functors}~\cref{item-assumption-finite-subcategories}
    are defined, then the above filtrations restrict directly to these finite subcategories.
\end{theorem}

\begin{proof}
    This is a combination of
    \cref{prop-st-filtrations,proposition:explicitation_partial_steinberg}.
    The version for cones follows by the discussion at the beginning of this subsection.
\end{proof}

\begin{remark}[D-modules]
    \cref{theorem:mackey-formula} also holds
    for all D-modules (\cref{eg-d-modules}),
    although \cref{assumption-six-functors} is not fully satisfied in this case.
    Indeed, as mentioned in \cref{eg-d-modules},
    all induction and restriction functors are still defined.
    The case of $\mathrm{Ind}'$ and $\mathrm{Res}'$
    follows from the proof of the theorem without modification,
    and the case of $\mathrm{Ind}$ and $\mathrm{Res}$
    then follows by taking adjunctions.
\end{remark}

\begin{theorem}
    \label{theorem:mackey-semisimple}
    In the situation of
    \cref{theorem:mackey-formula},
    suppose we are in either of the following cases:
    \begin{enumerate}
        \item
            $\mathcal{D}_\mathrm{c} (-)$ is the category of mixed Hodge modules
            as in \cref{eg-mhm},
            and $F \in \mathcal{D}_\mathrm{c} (\mathcal{X}_\lambda)$
            is a pure object; or
        \item
            $\mathcal{D}_\mathrm{c} (-)$ is the category of constructible sheaves
            as in \cref{eg-constructible-sheaves}
            with coefficients $\Lambda = \mathbb{Q}$,
            and $F \in \mathcal{D}_\mathrm{c} (\mathcal{X}_\lambda)$
            is the realization of a pure mixed Hodge module.
    \end{enumerate}
    Then we have
    \begin{align*}
        \mathrm{Res}_{\mu} \circ \mathrm{Ind}_{\lambda} (F)
        & \simeq
        \bigoplus_{i \in I}
        \mathrm{Ind}_{\lambda_i} \circ T^i_{\lambda, \mu} \circ \mathrm{Res}_{\mu_i} (F) \ ,
        \\
        \mathrm{Res}'_{\mu} \circ \mathrm{Ind}'_{\lambda} (F)
        & \simeq
        \bigoplus_{i \in I}
        \mathrm{Ind}'_{\lambda_i} \circ T^{i \prime}_{\lambda, \mu} \circ \mathrm{Res}'_{\mu_i} (F) \ .
    \end{align*}
\end{theorem}

\begin{proof}
    This follows from \Cref{theorem:mackey-formula}
    and the fact that induction, restriction, and Tate twists
    all preserve pure mixed Hodge modules and their weights,
    by \cref{prop-ind-preserves-purity,prop-res-preserves-purity}.
    It follows that the associated graded pieces are pure of the same weight,
    and by \textcite[Corollary~4.3.8]{tubach-2025-mixed},
    there are no non-trivial extensions between them,
    and we obtain (non-canonical) splittings as stated.
\end{proof}

%% file: quotient.tex
\subsection{Linear quotients}

Let~$k$ be a field,
let $G$ be a reductive group over~$k$
with a split maximal torus $T \subset G$,
and let $V$ be a finite-dimensional $G$-representation.

We illustrate the Mackey formula,
\cref{theorem:mackey-formula},
for the quotient stack $\mathcal{X} = V / G$,
which is smooth and satisfies \cref{assumption-stratification}.

Recall from \cref{def-cl}
that the component lattice of~$\mathcal{X}$ is
$\mathrm{CL} (V / G) \simeq \Lambda_T / W$,
where $\Lambda_T = \mathrm{Hom} (\mathbb{G}_\mathrm{m}, T)$
is the cocharacter lattice,
and $W = \mathrm{N}_G (T) / T$ is the Weyl group.
Thus, we may identify elements
$\lambda \in \mathrm{CL} (V / G)$
with cocharacters
$\lambda \colon \mathbb{G}_\mathrm{m} \to T$
up to Weyl conjugation.

For a cocharacter $\lambda \colon \mathbb{G}_\mathrm{m} \to T$,
the induction and restriction functors in \cref{def-ind-res}
are defined from the correspondence
\begin{equation*}
    V_\lambda / L_\lambda
    \overset{q_\lambda}{\longleftarrow}
    V_\lambda^+ / P_\lambda
    \overset{p_\lambda}{\longrightarrow}
    V / G \ ,
\end{equation*}
where $V_\lambda, V_\lambda^+ \subset V$
are the fixed and attracting loci of~$\lambda$,
given by the $\lambda$-weight~$0$ and~$\geq 0$ parts of~$V$;
$L_\lambda = \mathrm{Z}_G (\lambda) \subset G$
is the Levi subgroup, and
$P_\lambda =
\{g \in G \mid \lim_{t \to 0} \lambda(t) \, g \, \lambda(t)^{-1} \text{ exists}\} \subset G$
is the corresponding parabolic subgroup.
The map~$q_\lambda$ is induced by the projections
$V_\lambda^+ \twoheadrightarrow V_\lambda$
and $P_\lambda \twoheadrightarrow L_\lambda$,
and $p_\lambda$ by the inclusions
$V_\lambda^+ \hookrightarrow V$
and $P_\lambda \hookrightarrow G$.

Let $\lambda, \mu \colon \mathbb{G}_\mathrm{m} \to T$
be two cocharacters.
The set of relative positions of~$\lambda$ and~$\mu$ is
\begin{equation*}
    I = W_\lambda \backslash W / W_\mu \ ,
\end{equation*}
where $W_\lambda, W_\mu \subset W$
are the stabilizers of $\lambda, \mu$, respectively.
It is equipped with the Bruhat order~$\preceq$,
which we may extend arbitrarily to a total order~$\leq$.

For each $i \in I$,
let $w_i \in \mathrm{N}_G (T)$ be a lift of $i \in W_\lambda \backslash W / W_\mu$,
and we have the cocharacters
$\lambda_i = w_i^{-1} \, \lambda \, w_i$
and
$\mu_i = w_i \, \mu \, w_i^{-1}$.
The induction and restriction functors for
$\lambda_i$ and $\mu_i$
in the Mackey formula in \cref{theorem:mackey-formula}
are defined from the correspondences
\begin{gather*}
    \frac{V_{\lambda_i} \cap V_\mu}{L_{\lambda_i} \cap L_\mu}
    \ \overset{q^i_\lambda}{\longleftarrow} \
    \frac{V_{\lambda_i}^+ \cap V_\mu}{P_{\lambda_i} \cap L_\mu}
    \ \overset{p^i_\lambda}{\longrightarrow} \
    \frac{V_\mu}{L_\mu} \ ,
    \\
    \frac{V_\lambda \cap V_{\mu_i}}{L_\lambda \cap L_{\mu_i}}
    \ \overset{q^i_\mu}{\longleftarrow} \
    \frac{V_\lambda \cap V_{\mu_i}^+}{L_\lambda \cap P_{\mu_i}}
    \ \overset{p^i_\mu}{\longrightarrow} \
    \frac{V_\lambda}{L_\lambda} \ ,
\end{gather*}
respectively.
Note that we can also write the first line using
$(\lambda, \mu_i)$ in place of $(\lambda_i, \mu)$
and obtain isomorphic correspondences,
and similarly for the second line.

Next, we have the Steinberg stack
\begin{equation*}
    \calSt_{\lambda, \mu} =
    V_\lambda^+ / P_\lambda
    \underset{V / G}{\times}
    V_\mu^+ / P_\mu
    \simeq
    P_\lambda \backslash Z / P_\mu \ ,
    \quad
    Z = \{ (x, g, y) \in V_\lambda^+ \times G \times V_\mu^+ \mid x = g y \} \ ,
\end{equation*}
with the action on~$Z$ given by
$p (x, g, y) q = (p x, p g q, q^{-1} y)$
for $p \in P_\lambda$ and $q \in P_\mu$.

For each $i \in I$, we have the Bruhat stratum
\begin{equation*}
    \calSt_{\lambda, \mu}^i \simeq
    \frac{V_\lambda^+ \cap V_{\mu_i}^+}{P_\lambda \cap P_{\mu_i}}
    \longhookrightarrow
    \calSt_{\lambda, \mu} \ ,
\end{equation*}
where we also have
$P_{\mu_i} = w_i P_\mu w_i^{-1}$
as subgroups of~$G$.
The inclusion of this stratum is induced by the maps
$V_\lambda^+ \cap V_{\mu_i}^+ \to Z$
given by $v \mapsto (v, w_i, w_i^{-1} v)$,
and $P_\lambda \cap P_{\mu_i} \to P_\lambda \times P_\mu$
given by $p \mapsto (p, w_i^{-1} p^{-1} w_i)$.

The cocharacters $(\lambda, \mu_i) \colon \mathbb{G}_\mathrm{m}^2 \to T$
now induces a $\mathbb{Z}^2$-grading
$V = \bigoplus_{a, b \in \mathbb{Z}} V^i_{a, b}$.
The morphism~$h_{\lambda, \mu}^i$
in \cref{lem-affine-bundle} is now
\begin{equation*}
    h_{\lambda, \mu}^i \colon
    \calSt_{\lambda, \mu}^i \simeq
    \frac{V^i_{\geq 0, \, \geq 0}}{P_\lambda \cap P_{\mu_i}}
    \longrightarrow
    \frac{\bigoplus_{a, b \geq 0, \, ab = 0} V^i_{a, b}}
    {(L_\lambda \cap P_{\mu_i}) \times_{L_\lambda \cap L_{\mu_i}} (P_\lambda \cap L_{\mu_i})}{}
    \ ,
\end{equation*}
where the right-hand side represents the composition
$\mathrm{Ind}_{\lambda_i} \circ \mathrm{Res}_{\mu_i}$.
The difference between the two sides
precisely captures the twist $T^i_{\lambda, \mu}$
in the Mackey formula.
Since the projection
$V^i_{\geq 0, \, \geq 0} \to \bigoplus_{a, b \geq 0, \, ab = 0} V^i_{a, b}$
has kernel $V^i_{>0, >0}$,
and the projection
$P_\lambda \cap P_{\mu_i} \to
(L_\lambda \cap P_{\mu_i}) \times_{L_\lambda \cap L_{\mu_i}} (P_\lambda \cap L_{\mu_i})$
has kernel the unipotent group
$N_\lambda \cap N_{\mu_i}$,
where $N_\lambda \subset P_\lambda$
and $N_{\mu_i} \subset P_{\mu_i}$
are the unipotent radicals.
It follows that $h^i_{\lambda, \mu}$ is smooth locally the quotient of
an affine space fibration with fibre~$V^i_{> 0, \, > 0}$
by the unipotent group $N_\lambda \cap N_{\mu_i}$,
recovering the conclusion of \cref{lem-affine-bundle}.
We also see that
\begin{equation*}
    \dim h^i_{\lambda, \mu} =
    \dim V^i_{> 0, \, > 0}
    - \dim \mathfrak{g}^i_{> 0, \, > 0} \ ,
\end{equation*}
which appears in the twist in the Mackey formula.

We now have more explicit descriptions of
many ingredients of \cref{theorem:mackey-formula}.
We mention special cases that are of interest below.

\begin{example}[Adjoint quotients]
    \label{eg-mackey-adjoint}
    In the case when $V = \mathfrak{g}$ is the adjoint representation of~$G$,
    we give a more explicit description of
    the Mackey formula in \cref{theorem:mackey-formula}.
    The Mackey formula for D-modules in this case is studied by
    \textcite{gunningham2018generalized}.

    In this case, for a parabolic subgroup $P$ of~$G$
    with Levi quotient~$L$,
    denote their respective Lie algebras by
    $\mathfrak{p}, \mathfrak{l}$.
    From the correspondence
    \[
        \mathfrak{l}/L
        \overset{q}{\longleftarrow}
        \mathfrak{p}/P
        \overset{p}{\longrightarrow}
        \mathfrak{g}/G \ ,
    \]
    we define the induction and restriction functors
    $\mathrm{Ind}_{P, \, L}^G = p_! \, q^*$
    and
    $\mathrm{Res}_{P, \, L}^G = q_! \, p^*$.

    For parabolics $P, Q \subset G$ containing~$T$
    with corresponding Levi subgroups $L, M \subset G$,
    we have the Steinberg functor
    \begin{equation*}
        \mathrm{St}_{P, \, Q} = \mathrm{Res}_{Q, \, M}^G \circ \mathrm{Ind}_{P, \, L}^G
        \colon \mathcal{D} (\mathfrak{l}/L)
        \longrightarrow \mathcal{D} (\mathfrak{m}/M) \ .
    \end{equation*}
    The Mackey formula gives a filtration of it by the functors
    $\mathrm{St}_{P, \, Q}^w$ where $w \in W_P \backslash W / W_Q$,
    with
    \[
        \mathrm{St}_{P, \, Q}^w \simeq
        \mathrm{Ind}^M_{\smash{M \cap P^{\dot{w}}, \, M \cap L^{\dot{w}}}} \circ
        \dot{w}^* \circ
        \mathrm{Res}^L_{\smash{L \cap {}^{\dot{w}}Q, \, L \cap {}^{\dot{w}}M}}
        \ ,
    \]
    where $\dot{w} \in \mathrm{N}_G (T)$ denotes a lift of~$w$,
    we denote
    $(-)^{\dot{w}} = \dot{w}^{-1} (-) \dot{w}$
    and
    $^{\dot{w}} (-) = \dot{w} (-) \dot{w}^{-1}$,
    and $\dot{w}^*$ denotes the equivalence induced by the isomorphism
    \begin{equation*}
        \dot{w}^{-1} (-) \dot{w} \colon
        L \cap {}^{\dot{w}} M \longsimto M \cap L^{\dot{w}} \ .
    \end{equation*}
    Note that there is no Tate twist,
    since $\dim h^w_{\lambda, \mu} = 0$
    in the notation above.

    See also \cite[\S 1.2]{gunningham2018generalized}
    for this formula for D-modules.
    This formalism also appears in the study of character sheaves
    on reductive Lie algebras,
    as in \cite{mirkovic2004character} and \cite{zhou2024character}.

    An analogous description also holds for the adjoint quotient
    $G / G$
    in place of $\mathfrak{g} / G$.
\end{example}

\begin{example}[Quiver representations]
\label{example:quiver-representations}
    Let $Q=(Q_0,Q_1)$ be a finite quiver with set of vertices $Q_0$ and set of arrows $Q_1$. We let $s,t\colon Q_1\rightarrow Q_0$ be the source and target maps. For $\mathbf{d}\in\mathbb{N}^{Q_0}$, the stack of $\mathbf{d}$-dimensional $Q$-representation of dimension vector $\mathbf{d}$ is given by
    \[
        \mathcal{X}_{Q, \, \mathbf{d}} =
        V_{Q, \, \mathbf{d}}/\mathrm{GL}_{\mathbf{d}} \ ,
        \qquad
        V_{Q, \, \mathbf{d}} =
        \prod_{\alpha\in Q_1}
        \mathrm{Hom} (k^{\mathbf{d}_{s(\alpha)}}, k^{\mathbf{d}_{t(\alpha)}}) \ ,
        \qquad
        \quad\mathrm{GL}_{\mathbf{d}}
        = \prod_{i\in Q_0} \mathrm{GL}_{\mathbf{d}_i} \ .
    \]
    For dimension vectors $\mathbf{d}, \mathbf{e}\in\mathbb{N}^{Q_0}$,
    consider the cocharacter
    $\lambda_{\mathbf{d}, \mathbf{e}} \colon \mathbb{G}_\mathrm{m} \to \mathrm{GL}_{\mathbf{d} + \mathbf{e}}$
    given by
    \[
        \lambda_{\mathbf{d}, \mathbf{e}} (t) =
        \bigl(
            t \, \mathrm{Id}_{\smash{k^{\mathbf{d}_i}}}
            \oplus \mathrm{Id}_{\smash{k^{\mathbf{e}_i}}}
        \bigr)_{i \in Q_0}
        \in \mathrm{GL}_{\mathbf{d} + \mathbf{e}} \ .
    \]
    The induction functor for $\lambda_{\mathbf{d}, \mathbf{e}}$
    gives the multiplication in various types of Hall algebras,
    and these induction and restriction functors are studied by Lusztig
    in his geometric quantization of quantum groups,
    with the definition of the canonical basis
    \cite[\S9.2]{lusztig1993introduction}.
    See also \cite{schiffmann2012lecturescanonical} for an exposition.

    The Mackey formula categorifies the fact that
    the induction and restriction functors
    acting on split Grothendieck groups
    $\mathrm{H}_Q = \bigoplus_{\mathbf{d}\in\mathbb{N}^{Q_0}}
    \mathrm{K}_\oplus (\mathcal{D}'_{\mathrm{c}} (\mathcal{X}_{Q, \, \mathbf{d}}))$
    form a (twisted) bialgebra,
    where we take $k = \mathbb{C}$,
    and $\mathcal{D}'_{\mathrm{c}} (\mathcal{X}_{Q, \, \mathbf{d}})$
    is Lusztig's subcategory of constructible sheaves,
    which is closed under induction and restriction,
    and on which the Mackey filtration splits as in
    \cref{theorem:mackey-semisimple}.
    On the split Grothendieck groups,
    induction gives a multiplication
    $m\colon \mathrm{H}_Q\otimes\mathrm{H}_Q\rightarrow\mathrm{H}_Q$
    and restriction gives a comultiplication
    $\Delta\colon\mathrm{H}_Q\rightarrow\mathrm{H}_Q\otimes\mathrm{H}_Q$.
    Our Mackey formula translates to the identity
    \begin{equation*}
        \Delta \circ m =
        (m \otimes m) \circ T \circ (\Delta \otimes \Delta) \ ,
    \end{equation*}
    where
    $T (a \otimes b \otimes c \otimes d) =
    q^{\chi (|d|, |a|)} (a \otimes c \otimes b \otimes d)$
    for homogeneous elements $a, b, c, d \in \mathrm{H}_Q$,
    where~$\chi$ is the Euler form of~$Q$,
    and~$q$ corresponds to the twist~$\langle 1 \rangle$.
    To obtain an actual twisted bialgebra from this,
    one needs to choose a different normalization,
    such as defining a new coproduct
    $\tilde{\Delta}\colon\mathrm{H}_Q\rightarrow\mathrm{H}_Q\otimes\mathrm{H}_Q$
    by
    $\tilde{\Delta}_{\mathbf{d}, \mathbf{e}} =
    q^{-\chi (\mathbf{e}, \mathbf{d})} \Delta_{\mathbf{d}, \mathbf{e}}$
    for $\mathbf{d}, \mathbf{e}\in\mathbb{N}^{Q_0}$,
    and from the previous identity, one can deduce that
    \begin{equation*}
        \tilde{\Delta} \circ m =
        (m \otimes m) \circ
        (1 \otimes S \otimes 1)
        \circ (\tilde{\Delta} \otimes \tilde{\Delta}) \ ,
    \end{equation*}
    where $S (b \otimes c) = q^{-\chi (|b|, |c|)} (c \otimes b)$
    for homogeneous elements $b, c \in \mathrm{H}_Q$,
    which is the required twisted bialgebra compatibility.
\end{example}

\subsection{Semistable \texorpdfstring{\textit{G}}{G}-bundles on a curve}
\label{subsec-bunss}

Let~$k$ be a field,
and let~$C$ be a geometrically connected smooth projective curve over~$k$
of genus~$g$.
Let $G$ be a split reductive group over~$k$,
and consider the moduli stack~$\mathcal{X} = \calBun_G^\mathrm{ss} (C)_d$
of semistable $G$-bundles on~$C$ of degree~$d$,
where $d \in \uppi_1 (G_{\smash{\bar{k}}})$.

We check the assumptions on~$\mathcal{X}$ in
\cref{theorem:stratification} as follows.
If $\operatorname{char} k = 0$,
then~$\mathcal{X}$ admits a good moduli space by
\cite[Theorem~8.1]{alper-halpern-leistner-heinloth-2023},
and a quadratic norm on $\mathrm{CL} (\mathcal{X})$
can be given by a global quotient presentation.
If $\operatorname{char} k > 0$,
then~$\mathcal{X}$ admits an adequate moduli space by
\cite[Theorem~2.1 and Lemma~4.8]{weissmann-zhang-2025},
and it admits a global quotient presentation $Y / \mathrm{GL}_n$ by the proof of
\cite[Proposition~3.28]{weissmann-zhang-2025}
(although the proof assumes $k = \bar{k}$,
it gives a morphism $\mathcal{X} \to \mathrm{BGL}_n$ defined over~$k$,
and it being representable over~$\bar{k}$ implies that over~$k$).

Let $P \subset G$ be a parabolic subgroup
with the corresponding Levi subgroup $L \subset G$,
and suppose that~$d$ is induced by an \emph{admissible degree}
$d_L \in \uppi_1 (L_{\smash{\bar{k}}}) \simeq \uppi_1 (P_{\smash{\bar{k}}})$
in the sense of
\cite[\S 10.3.1]{bu-davison-ibanez-nunez-kinjo-padurariu},
which is uniquely determined by~$d$ when it exists.
Consider the correspondence
\begin{equation*}
    \calBun_L^\mathrm{ss} (C)_{d_L}
    \overset{q}{\longleftarrow}
    \calBun_P^\mathrm{ss} (C)_{d_L}
    \overset{p}{\longrightarrow}
    \calBun_G^\mathrm{ss} (C)_d \ ,
\end{equation*}
with the meaning of $\calBun_P^\mathrm{ss} (C)$
as in \cref{eg-steinberg-bun}.
Again, we have the induction and restriction functors
$\mathrm{Ind}_{P, \, L}^G = p_! \, q^*$
and
$\mathrm{Res}_{P, \, L}^G = q_! \, p^*$.

For parabolics $P, Q \subset G$ containing~$T$
with corresponding Levi subgroups $L, M \subset G$,
and for a degree $d \in \uppi_1 (G_{\smash{\bar{k}}})$
for which there exist compatible degrees
$d_L \in \uppi_1 (L_{\smash{\bar{k}}})$
and $d_M \in \uppi_1 (M_{\smash{\bar{k}}})$,
\cref{theorem:mackey-formula}
gives a filtration of the Steinberg functor
\begin{equation*}
    \mathrm{St}_{P, \, Q} =
    \mathrm{Res}_{Q, \, M}^G \circ \mathrm{Ind}_{P, \, L}^G
    \colon \mathcal{D} (\calBun_L^\mathrm{ss} (C)_{d_L})
    \longrightarrow \mathcal{D} (\calBun_M^\mathrm{ss} (C)_{d_M})
\end{equation*}
with associated graded pieces
\[
    \mathrm{St}_{P, \, Q}^w \simeq
    \mathrm{Ind}^M_{\smash{M \cap P^{\dot{w}}, \, M \cap L^{\dot{w}}}} \circ
    \dot{w}^* \langle \theta_w \rangle \circ
    \mathrm{Res}^L_{\smash{L \cap {}^{\dot{w}}Q, \, L \cap {}^{\dot{w}}M}}
    \ ,
\]
with notations analogous to \cref{eg-mackey-adjoint},
except that the term $\mathrm{St}_{P, \, Q}^w$
is present in the filtration only if there exists a degree
$d_{L, \, M} \in \uppi_1 (L_{\smash{\bar{k}}} \cap {}^{\dot{w}} M_{\smash{\bar{k}}})$
compatible with both~$d_L$ and~$d_M$.
We define
$\theta_w = (1-g) \dim (\mathfrak{n}_{\mathfrak{p}} \cap {}^{\dot{w}} \mathfrak{n}_{\mathfrak{q}})$,
where $\mathfrak{n}_{\mathfrak{p}} \subset \mathfrak{p}$
and $\mathfrak{n}_{\mathfrak{q}} \subset \mathfrak{q}$
are the nilradicals,
which can be deduced from \cref{eq-dim-h}.

As mentioned in the introduction,
it would be interesting to extend this formula
from the semistable locus $\calBun_G^\mathrm{ss} (C)$
to the full stack $\calBun_G (C)$,
possibly with modified induction and restriction functors
and a more sophisticated twist functor in place of
$\dot{w}^* \langle \theta_w \rangle$,
which would give a compatibility between the full automorphic
Eisenstein and constant term functors
in geometric Langlands correspondences.

%% file: main.bib
@article{Antor2024,
  author  = {Antor, Jonas},
  title   = {Formality in the Deligne--Langlands correspondence},
  journal = {Int. Math. Res. Not.},
  volume  = {2024},
  number  = {7},
  pages   = {6043--6072},
  year    = {2024},
  doi     = {10.1093/imrn/rnad297},
  eprint  = {2302.11010},
  archivePrefix = {arXiv}
}

@article{BravermanGaitsgory2002,
    author       = {Braverman, A. and Gaitsgory, D.},
    title        = {Geometric Eisenstein series},
    volume       = {150},
    doi          = {10.1007/s00222-002-0237-8},
    pages        = {287--384},
    number       = {2},
    journaltitle = {Invent. Math.},
    eprinttype   = {arxiv},
    eprint       = {math/9912097},
    year         = {2002},
}

@article{bu2025intrinsic,
    author       = {Bu, Chenjing and Halpern-Leistner, Daniel and Ibáñez Núñez, Andrés and Kinjo, Tasuki},
    title        = {Intrinsic Donaldson–Thomas theory. I: Component lattices of stacks},
    version      = {preprint~v2},
    eprinttype   = {arxiv},
    eprint       = {2502.13892v2},
    date         = {2025-09-11},
}

@book{chriss1997representation,
  title={Representation theory and complex geometry},
  author={Chriss, N. and Ginzburg, V.},
  volume={42},
  year={1997},
  publisher={Springer}
}

@article{deligne1976representations,
    author       = {Deligne, P. and Lusztig, G.},
    title        = {Representations of reductive groups over finite fields},
    volume       = {103},
    doi          = {10.2307/1971021},
    pages        = {103--161},
    number       = {1},
    journaltitle = {Ann. Math.},
    year         = {1976},
}

@article{drinfeld2016geometric,
    author       = {Drinfeld, Vladimir and Gaitsgory, Dennis},
    title        = {Geometric constant term functor(s)},
    volume       = {22},
    doi          = {10.1007/s00029-016-0269-3},
    pages        = {1881--1951},
    number       = {4},
    journaltitle = {Sel. Math.},
    eprinttype   = {arxiv},
    eprint       = {1311.2071},
    year         = {2016},

}

@article{gunningham2018generalized,
    author       = {Gunningham, Sam},
    title        = {Generalized Springer theory for D-modules on a reductive Lie algebra},
    version      = {revised version, preprint~v4},
    eprinttype   = {arxiv},
    eprint       = {1510.02452v4},
    date         = {2021-01-30},
    addendum     = {original Sel.\ Math.\ 24.5 (2018), 4223--4277, doi: \href{https://doi.org/10.1007/s00029-018-0443-x}{\texttt{10.1007/s00029-018-0443-x}}; correction Sel.\ Math.\ 27.4 (2021), 62, 10~pp., doi: \href{https://doi.org/10.1007/s00029-021-00684-9}{\texttt{10.1007/s00029-021-00684-9}}},
}

@article{halpern-leistner-instability,
    author       = {Halpern-Leistner, Daniel},
    title        = {On the structure of instability in moduli theory},
    version      = {preprint~v5},
    eprinttype   = {arxiv},
    eprint       = {1411.0627v5},
    date         = {2022-02-04},
}

@article{Hennecart2024,
    author       = {Hennecart, Lucien},
    title        = {Perverse sheaves with nilpotent singular support on the stack of coherent sheaves on an elliptic curve},
    journaltitle = {Transform. Groups},
    volume       = {29},
    number       = {1},
    pages        = {103--142},
    year         = {2024},
    doi          = {10.1007/s00031-022-09749-1},
    eprinttype   = {arxiv},
    eprint       = {2101.03813},
}

@book{humphreys-1990-coxeter,
    author       = {Humphreys, James E.},
    title        = {Reflection groups and Coxeter groups},
    series       = {Camb. Stud. Adv. Math.},
    number       = {29},
    publisher    = {Camb. Univ. Press},
    year         = {1990},
    doi          = {10.1017/cbo9780511623646},
}

@article{kinjo2024cohomological,
    author       = {Kinjo, Tasuki and Park, Hyeonjun and Safronov, Pavel},
    title        = {Cohomological Hall algebras for $3$-Calabi--Yau categories},
    version      = {preprint~v2},
    eprinttype   = {arxiv},
    eprint       = {2406.12838v2},
    date         = {2026-01-24},
}

@article{KhovanovLauda2009,
    author       = {Khovanov, Mikhail and Lauda, Aaron D.},
    title        = {A diagrammatic approach to categorification of quantum groups. I},
    volume       = {13},
    doi          = {10.1090/s1088-4165-09-00346-x},
    pages        = {309--347},
    journaltitle = {Represent. Theory},
    eprinttype   = {arxiv},
    eprint       = {0803.4121},
    year         = {2009},
}

@book{letellier2004fourier,
    author       = {Letellier, Emmanuel},
    title        = {Fourier transforms of invariant functions on finite reductive Lie algebras},
    doi          = {10.1007/b104209},
    series       = {Lect. Notes Math.},
    number       = {1859},
    publisher    = {Springer},
    year         = {2005},
}

@article{lusztig1991quivers,
    author       = {Lusztig, George},
    title        = {Quivers, perverse sheaves, and quantized enveloping algebras},
    volume       = {4},
    doi          = {10.1090/s0894-0347-1991-1088333-2},
    pages        = {365--421},
    number       = {2},
    journaltitle = {J. Am. Math. Soc.},
    year         = {1991},
}

@book {lusztig1993introduction,
    AUTHOR = {Lusztig, George},
     TITLE = {Introduction to quantum groups},
    SERIES = {Prog. Math.},
    number = {110},
 PUBLISHER = {Birkh\"{a}user},
      YEAR = {1993},
      ISBN = {0-8176-3712-5}
}

@incollection{lusztig2006character,
    author       = {Lusztig, George},
    editor       = {Etingof, Pavel and others},
    title        = {Character sheaves and generalizations},
    doi          = {10.1007/0-8176-4467-9_12},
    series       = {Prog. Math.},
    pages        = {443--455},
    number       = {244},
    booktitle    = {The unity of mathematics},
    publisher    = {Birkhäuser},
    eprinttype   = {arxiv},
    eprint       = {math/0309134},
    year         = {2007},
}

@incollection{MarsSpringer1989,
    author       = {Mars, J. G. M. and Springer, T. A.},
    title        = {Character sheaves},
    booktitle    = {Orbites unipotentes et repr{\'e}sentations III. Orbites et faisceaux pervers},
    series       = {Ast{\'e}risque},
    number       = {173--174},
    pages        = {111--198},
    publisher    = {Soc. Math. Fr.},
    year         = {1989},
    url          = {https://numdam.org/item/AST_1989__173-174__111_0/},
}

@article{mirkovic2004character,
    author       = {Mirković, Ivan},
    title        = {Character sheaves on reductive Lie algebras},
    volume       = {4},
    pages        = {897--910},
    number       = {4},
    journaltitle = {Mosc. Math. J.},
    eprinttype   = {arxiv},
    eprint       = {math/0411532},
    year         = {2004},
}

@article{Rouquier2012,
    author       = {Rouquier, Raphaël},
    title        = {Quiver Hecke algebras and $2$-Lie algebras},
    volume       = {19},
    doi          = {10.1142/s1005386712000247},
    pages        = {359--410},
    number       = {2},
    journaltitle = {Algebra Colloq.},
    eprinttype   = {arxiv},
    eprint       = {1112.3619},
    year         = {2012},
}

@article{Schiffmann2006,
    author       = {Schiffmann, Olivier},
    title        = {Canonical bases and moduli spaces of sheaves on curves},
    journaltitle = {Invent. Math.},
    volume       = {165},
    number       = {3},
    pages        = {453--524},
    year         = {2006},
    doi          = {10.1007/s00222-005-0495-3},
}

@article{Schiffmann2012,
    author       = {Schiffmann, Olivier},
    title        = {On the {Hall} algebra of an elliptic curve. {II}},
    journaltitle = {Duke Math. J.},
    volume       = {161},
    number       = {9},
    pages        = {1711--1750},
    year         = {2012},
    doi          = {10.1215/00127094-1593362},
}

@incollection {schiffmann2012lecturescanonical,
    AUTHOR = {Schiffmann, Olivier},
    editor = {Brion, Michel},
     TITLE = {Lectures on canonical and crystal bases of {H}all algebras},
 BOOKTITLE = {Geometric methods in representation theory. {II}},
    SERIES = {S\'{e}min. Congr.},
    number = {24.2},
     PAGES = {143--259},
 PUBLISHER = {Soc. Math. Fr.},
      YEAR = {2012},
      ISBN = {978-2-85629-361-4},
   MRCLASS = {17B37 (14D23 14D24 14L30)},
  MRNUMBER = {3202708},
MRREVIEWER = {Kailash\ C.\ Misra},
}

@article{zhou2024character,
    author       = {Zhou, Tong},
    title        = {Character sheaves on reductive Lie algebras in positive characteristic},
    volume       = {2024},
    doi          = {10.1093/imrn/rnae221},
    pages        = {13827--13839},
    number       = {22},
    journaltitle = {Int. Math. Res. Not.},
    eprinttype   = {arxiv},
    eprint       = {2404.19210},
    year         = {2024},
}

@article{ScholzeSixFunctor,
    author       = {Scholze, Peter},
    title        = {Six-functor formalisms},
    version      = {preprint~v2},
    eprinttype   = {arxiv},
    eprint       = {2510.26269v2},
    date         = {2026-01-22},
}

@article{drinfeld-gaitsgory-2013-finiteness,
    author       = {Drinfeld, Vladimir and Gaitsgory, Dennis},
    title        = {On some finiteness questions for algebraic stacks},
    volume       = {23},
    doi          = {10.1007/s00039-012-0204-5},
    pages        = {149--294},
    number       = {1},
    journaltitle = {Geom. Funct. Anal.},
    eprinttype   = {arxiv},
    eprint       = {1108.5351},
    year         = {2013},
}

@article{drinfeld-gaitsgory-2014-braden,
    author       = {Drinfeld, Vladimir and Gaitsgory, Dennis},
    title        = {On a theorem of Braden},
    volume       = {19},
    doi          = {10.1007/s00031-014-9267-8},
    pages        = {313--358},
    number       = {2},
    journaltitle = {Transform. Groups},
    eprinttype   = {arxiv},
    eprint       = {1308.3786},
    year         = {2014},
}

@article{drinfeld-gaitsgory-2015-compact,
    author       = {Drinfeld, Vladimir and Gaitsgory, Dennis},
    title        = {Compact generation of the category of D-modules on the stack of $G$-bundles on a curve},
    volume       = {3},
    doi          = {10.4310/cjm.2015.v3.n1.a2},
    pages        = {19--125},
    number       = {1},
    journaltitle = {Camb. J. Math.},
    eprinttype   = {arxiv},
    eprint       = {1112.2402},
    year         = {2015},
}

@article{sandvik,
    author       = {Sandvik, Colton},
    title        = {Endoscopy for modular Hecke categories},
    version      = {preprint~v2},
    eprinttype   = {arxiv},
    eprint       = {2508.10214v2},
    date         = {2026-04-16},
}

@article{tubach-2025-mixed,
    author       = {Tubach, Swann},
    title        = {Mixed Hodge modules on stacks},
    volume       = {13},
    doi          = {10.1017/fms.2025.10122},
    pages        = {e175, 29~pp.},
    journaltitle = {Forum Math. Sigma},
    eprinttype   = {arxiv},
    eprint       = {2407.02256},
    year         = {2025},
}

@article{agkrrv,
    author       = {Arinkin, D. and Gaitsgory, Dennis and Kazhdan, D. and Raskin, Sam and Rozenblyum, N. and Varshavsky, Y.},
    title        = {The stack of local systems with restricted variation and geometric Langlands theory with nilpotent singular support},
    version      = {preprint~v2},
    eprinttype   = {arxiv},
    eprint       = {2010.01906v2},
    date         = {2022-04-06},
}

@article{koppensteiner-2020-hochschild,
    author       = {Koppensteiner, Clemens},
    title        = {Hochschild cohomology of torus equivariant D-modules},
    volume       = {2020},
    doi          = {10.1093/imrn/rny206},
    pages        = {6391--6420},
    number       = {19},
    journaltitle = {Int. Math. Res. Not.},
    eprinttype   = {arxiv},
    eprint       = {1506.07370},
    year         = {2020},
}

@article{liu-zheng,
    author       = {Liu, Yifeng and Zheng, Weizhe},
    title        = {Enhanced six operations and base change theorem for higher Artin stacks},
    version      = {preprint~v4},
    eprinttype   = {arxiv},
    eprint       = {1211.5948v4},
    date         = {2024-12-17},
}

@article{hemo-richarz-scholbach-2023,
    author       = {Hemo, Tamir and Richarz, Timo and Scholbach, Jakob},
    title        = {Constructible sheaves on schemes},
    volume       = {429},
    doi          = {10.1016/j.aim.2023.109179},
    pages        = {109179, 46~pp.},
    journaltitle = {Adv. Math.},
    eprinttype   = {arxiv},
    eprint       = {2305.18131},
    year         = {2023},
}

@article{alper-hall-rydh-2026-etale,
    author       = {Alper, Jarod and Hall, Jack and Rydh, David},
    title        = {The étale local structure of algebraic stacks},
    volume       = {59},
    doi          = {10.24033/asens.2637},
    pages        = {125--198},
    number       = {1},
    journaltitle = {Ann. Sci. Éc. Norm. Supér.},
    eprinttype   = {arxiv},
    eprint       = {1912.06162},
    year         = {2026},
}

@article{alper-2013-good,
    author       = {Alper, Jarod},
    title        = {Good moduli spaces for Artin stacks},
    volume       = {63},
    doi          = {10.5802/aif.2833},
    pages        = {2349--2402},
    number       = {6},
    journaltitle = {Ann. Inst. Fourier},
    eprinttype   = {arxiv},
    eprint       = {0804.2242},
    year         = {2013},
}

@article{alper-2014-adequate,
    author       = {Alper, Jarod},
    title        = {Adequate moduli spaces and geometrically reductive group schemes},
    volume       = {1},
    doi          = {10.14231/ag-2014-022},
    pages        = {489--531},
    number       = {4},
    journaltitle = {Algebr. Geom.},
    eprinttype   = {arxiv},
    eprint       = {1005.2398},
    year         = {2014},
}

@incollection{conrad-2014-reductive,
    author       = {Conrad, Brian},
    bookauthor   = {Brochard, Sylvain and Conrad, Brian and Oesterlé, Joseph},
    title        = {Reductive group schemes},
    series       = {Panor. Synth.},
    pages        = {93--444},
    number       = {42–43},
    booktitle    = {Autour des schémas en groupes. École d’Été “Schémas en groupes”. Volume I},
    publisher    = {Soc. Math. Fr.},
    year         = {2014},
}

@article{ibanez-nunez-finiteness,
    author       = {Ibáñez Núñez, Andrés},
    title        = {Finiteness of good moduli spaces of graded points},
    version      = {preprint},
    url          = {https://www.math.columbia.edu/~ibaneznunez/documents/finiteness-gms-graded.pdf},
    date         = {2025-03},
}

@article{luna-1975-adherences,
    author       = {Luna, D.},
    title        = {Adhérences d’orbite et invariants},
    volume       = {29},
    doi          = {10.1007/bf01389851},
    pages        = {231--238},
    journaltitle = {Invent. Math.},
    year         = {1975},
}

@article{hansen-mann-langlands,
    author       = {Hansen, David and Mann, Lucas},
    title        = {The categorical local Langlands conjecture},
    version      = {preprint~v2},
    eprinttype   = {arxiv},
    eprint       = {2606.00983v2},
    date         = {2026-06-04},
}

@article{bezrukavnikov-2016-aha,
    author       = {Bezrukavnikov, Roman},
    title        = {On two geometric realizations of an affine Hecke algebra},
    volume       = {123},
    doi          = {10.1007/s10240-015-0077-x},
    pages        = {1--67},
    journaltitle = {Publ. Math., Inst. Hautes Étud. Sci.},
    eprinttype   = {arxiv},
    eprint       = {1209.0403},
    year         = {2016},
}

@article{bu-davison-ibanez-nunez-kinjo-padurariu,
    author       = {Bu, Chenjing and Davison, Ben and Ibáñez Núñez, Andrés and Kinjo, Tasuki and Pădurariu, Tudor},
    title        = {Cohomology of symmetric stacks},
    version      = {preprint~v2},
    eprinttype   = {arxiv},
    eprint       = {2502.04253v2},
    date         = {2025-05-31},
}

@article{braden-2003-hyperbolic,
    author       = {Braden, Tom},
    title        = {Hyperbolic localization of intersection cohomology},
    volume       = {8},
    doi          = {10.1007/s00031-003-0606-4},
    pages        = {209--216},
    number       = {3},
    journaltitle = {Transform. Groups},
    eprinttype   = {arxiv},
    eprint       = {math/0202251},
    year         = {2003},
}

@article{mayeux-2026-magnetism,
    author       = {Mayeux, Arnaud},
    title        = {Algebraic magnetism},
    volume       = {113},
    doi          = {10.1007/s00233-026-10645-3},
    pages        = {126--201},
    number       = {1},
    journaltitle = {Semigroup Forum},
    eprinttype   = {arxiv},
    eprint       = {2203.08231},
    year         = {2026},
}

@article{alper-halpern-leistner-heinloth-2023,
    author       = {Alper, Jarod and Halpern-Leistner, Daniel and Heinloth, Jochen},
    title        = {Existence of moduli spaces for algebraic stacks},
    volume       = {234},
    doi          = {10.1007/s00222-023-01214-4},
    pages        = {949--1038},
    number       = {3},
    journaltitle = {Invent. Math.},
    eprinttype   = {arxiv},
    eprint       = {1812.01128},
    year         = {2023},
}

@article{weissmann-zhang-2025,
    author       = {Weißmann, Dario and Zhang, Xucheng},
    title        = {A stacky approach to identifying the semistable locus of bundles},
    doi          = {10.14231/ag-2025-008},
    pages        = {262--298},
    journaltitle = {Algebr. Geom.},
    eprinttype   = {arxiv},
    eprint       = {2302.09245},
    year         = {2025},
}

@article{kapranov-schechtman-schiffmann-yuan,
    author       = {Kapranov, Mikhail and Schechtman, Vadim and Schiffmann, Olivier and Yuan, Jiangfan},
    title        = {The Langlands formula and perverse sheaves},
    version      = {preprint~v1},
    eprinttype   = {arxiv},
    eprint       = {2412.01638v1},
    date         = {2024-12-02},
}

@article{dyckerhoff-wedrich,
    author       = {Dyckerhoff, Tobias and Wedrich, Paul},
    title        = {Perverse schobers of Coxeter type A},
    version      = {preprint~v1},
    eprinttype   = {arxiv},
    eprint       = {2504.08496v1},
    date         = {2025-04-11},
}

@article{kapranov-schechtman-2022-induction,
    author       = {Kapranov, Mikhail and Schechtman, Vadim},
    title        = {Parabolic induction and perverse sheaves on $W \backslash \mathfrak{h}$},
    volume       = {398},
    doi          = {10.1016/j.aim.2022.108200},
    pages        = {108200, 52 pp.},
    journaltitle = {Adv. Math.},
    eprinttype   = {arxiv},
    eprint       = {2006.04285},
    year         = {2022},
}

@article{kapranov-schechtman-schobers,
    author       = {Kapranov, Mikhail and Schechtman, Vadim},
    title        = {Perverse schobers},
    version      = {preprint~v2},
    eprinttype   = {arxiv},
    eprint       = {1411.2772v2},
    date         = {2015-11-18},
}

@article{glc-v,
    author       = {Gaitsgory, Dennis and Raskin, Sam},
    title        = {Proof of the geometric Langlands conjecture. V: The multiplicity one theorem},
    version      = {preprint~v3},
    eprinttype   = {arxiv},
    eprint       = {2409.09856v3},
    date         = {2026-01-16},
}

@article{glc-i,
    author       = {Gaitsgory, Dennis and Raskin, Sam},
    title        = {Proof of the geometric Langlands conjecture. I: Construction of the functor},
    version      = {preprint~v3},
    eprinttype   = {arxiv},
    eprint       = {2405.03599v3},
    date         = {2025-09-30},
}

@article{glc-iii,
    author       = {Campbell, Justin and Chen, Lin and Færgeman, Joakim and Gaitsgory, Dennis and Lin, Kevin and Raskin, Sam and Rozenblyum, Nick},
    title        = {Proof of the geometric Langlands conjecture. {III}: Compatibility with parabolic induction},
    version      = {preprint~v1},
    eprinttype   = {arxiv},
    eprint       = {2409.07051v1},
    date         = {2024-09-11},
}

@article{glc-ii,
    author       = {Arinkin, D. and Beraldo, D. and Campbell, J. and Chen, L. and Færgeman, J. and Gaitsgory, D. and Lin, K. and Raskin, S. and Rozenblyum, N.},
    title        = {Proof of the geometric Langlands conjecture. {II}: Kac–Moody localization and the {FLE}},
    version      = {preprint~v3},
    eprinttype   = {arxiv},
    eprint       = {2405.03648v3},
    date         = {2024-09-11},
}

@article{glc-iv,
    author       = {Arinkin, D. and Beraldo, D. and Chen, L. and Færgeman, J. and Gaitsgory, D. and Lin, K. and Raskin, S. and Rozenblyum, N.},
    title        = {Proof of the geometric Langlands conjecture. {IV}: Ambidexterity},
    version      = {preprint~v1},
    eprinttype   = {arxiv},
    eprint       = {2409.08670v1},
    date         = {2024-09-13},
}

@article{bu-padurariu-toda,
    author       = {Bu, Chenjing and Pădurariu, Tudor and Toda, Yukinobu},
    title        = {Semiorthogonal decompositions for stacks},
    version      = {preprint~v1},
    eprinttype   = {arxiv},
    eprint       = {2605.25976v1},
    date         = {2026-05-25},
}

@article{hennecart-mackey,
    author       = {Hennecart, Lucien},
    title        = {Cohomological Mackey formula for quotient stacks},
    version      = {preprint~v1},
    eprinttype   = {arxiv},
    eprint       = {2505.09483v1},
    date         = {2025-05-14},
}

@article{zhu-tame,
    author       = {Zhu, Xinwen},
    title        = {Tame categorical local Langlands correspondence},
    version      = {preprint~v1},
    eprinttype   = {arxiv},
    eprint       = {2504.07482v1},
    date         = {2025-04-10},
}

@article{lusztig-1985-character-i,
    author       = {Lusztig, George},
    title        = {Character sheaves. I},
    volume       = {56},
    doi          = {10.1016/0001-8708(85)90034-9},
    pages        = {193--237},
    number       = {3},
    journaltitle = {Adv. Math.},
    year         = {1985},
}

@article{lusztig-1985-character-ii,
    author       = {Lusztig, George},
    title        = {Character sheaves. {II}},
    volume       = {57},
    doi          = {10.1016/0001-8708(85)90064-7},
    pages        = {226--265},
    number       = {3},
    journaltitle = {Adv. Math.},
    year         = {1985},
}

@article{lusztig-1985-character-iii,
    author       = {Lusztig, George},
    title        = {Character sheaves. {III}},
    volume       = {57},
    doi          = {10.1016/0001-8708(85)90065-9},
    pages        = {266--315},
    number       = {3},
    journaltitle = {Adv. Math.},
    year         = {1985},
}

@article{lusztig-1986-character-iv,
    author       = {Lusztig, George},
    title        = {Character sheaves. {IV}},
    volume       = {59},
    doi          = {10.1016/0001-8708(86)90036-8},
    pages        = {1--63},
    number       = {1},
    journaltitle = {Adv. Math.},
    year         = {1986},
}

@article{lusztig-1986-character-v,
    author       = {Lusztig, George},
    title        = {Character sheaves. V},
    volume       = {61},
    doi          = {10.1016/0001-8708(86)90071-x},
    pages        = {103--155},
    number       = {2},
    journaltitle = {Adv. Math.},
    year         = {1986},
}

@article{VaragnoloVasserot2011,
    author       = {Varagnolo, M. and Vasserot, E.},
    title        = {Canonical bases and {KLR}-algebras},
    volume       = {659},
    doi          = {10.1515/crelle.2011.068},
    pages        = {67--100},
    journaltitle = {J. Reine Angew. Math.},
    eprinttype   = {arxiv},
    eprint       = {0901.3992},
    year         = {2011},
}

@article{Sauter2013,
    author       = {Sauter, Julia},
    title        = {Generalized quiver Hecke algebras},
    version      = {preprint~v2},
    eprinttype   = {arxiv},
    eprint       = {1306.3892v2},
    date         = {2013-07-03},
}
